\documentclass[10pt,reqno]{amsart}

\usepackage{amsthm, mathrsfs, amsmath, amstext, amsxtra, amsfonts, dsfont, amssymb}
\usepackage{lmodern}
\usepackage[dvipsnames]{xcolor}
\usepackage[colorlinks, linkcolor=blue, citecolor=red, urlcolor=RedViolet, pagebackref, hypertexnames=false]{hyperref}

\newcommand{\N}{\mathbb N}

\newcommand{\R}{\mathbb R}
\newcommand{\C}{\mathbb C}

\newcommand{\Gc}{\mathcal G}
\newcommand{\Hc}{\mathcal H}

\newcommand{\eps}{\epsilon}
\newcommand{\sct} {s_{\text{c}}}
\newcommand{\sce} {s_{\emph{c}}}
\newcommand{\alct} {\alpha_{\text{c}}}
\newcommand{\alce} {\alpha_{\emph{c}}}
\newcommand{\re}[1]{\mbox{Re}  #1} 
\newcommand{\im}[1]{\mbox{Im}  #1}

\newcommand{\defendproof}{\hfill $\Box$} 

\newtheorem{theorem}{Theorem}[section]
\newtheorem{lemma}[theorem]{Lemma} 
\newtheorem{proposition}[theorem]{Proposition}
\newtheorem{corollary}[theorem]{Corollary} 
\theoremstyle{definition}
\newtheorem{definition}[theorem]{Definition}
\newtheorem{remark}[theorem]{Remark}

\title[Blowup solutions $L^2$-supercritical NLFS]{On Blowup solutions to the focusing $L^2$-supercritical nonlinear fractional Schr\"odinger equation} 

\author[V. D. Dinh]{Van Duong Dinh}
\address[V. D. Dinh]{Institut de Math\'ematiques de Toulouse UMR5219, Universit\'e Toulouse CNRS, 31062 Toulouse Cedex 9, France and Department of Mathematics, HCMC University of Pedagogy, 280 An Duong Vuong, Ho Chi Minh, Vietnam}
\email{dinhvan.duong@math.univ-toulouse.fr}

\keywords{Nonlinear fractional Schr\"odinger equation; Blowup; Concentration; Limiting profile}
\subjclass[2010]{35B44, 35Q55}

\begin{document}

\maketitle
\begin{abstract}
	In this paper we study dynamical properties of blowup solutions to the focusing $L^2$-supercritical nonlinear fractional Schr\"odinger equation
	\[
	i\partial_t -(-\Delta)^s u = -|u|^\alpha u, \quad u(0) = u_0, \quad \text{on } [0,\infty) \times \R^d,
	\]
	where $d \geq 2, \frac{d}{2d-1} \leq s <1$, $\frac{4s}{d}<\alpha<\frac{4s}{d-2s}$ and $u_0 \in \dot{H}^{\sct} 
	\cap \dot{H}^s$ is radial with the critical Sobolev exponent $\sct$. To this end, we establish a compactness lemma related to the equation by means of the profile decomposition for bounded sequences in $\dot{H}^{\sct} \cap \dot{H}^s$. As a result, we obtain the $\dot{H}^{\sct}$-concentration and the limiting profile with critical $\dot{H}^{\sct}$-norm of blowup solutions with bounded $\dot{H}^{\sct}$-norm. 
\end{abstract}


\section{Introduction}
\setcounter{equation}{0}
In this paper, we consider the Cauchy problem for the focusing $L^2$-supercritical nonlinear fractional Schr\"odinger equation
\begin{align}
\left\{
\begin{array}{rcl}
i\partial_t u - (-\Delta)^s u &=& -|u|^{\alpha} u, \quad \text{on } [0,+\infty) \times \R^d, \\
u(0) &=& u_0, 
\end{array}
\right.
\label{focusing intercritical NLFS}
\end{align}
where $u:[0,+\infty) \times \R^d \rightarrow \C$, $s\in (0,1) \backslash \{1/2\}$ and $\alpha>0$. The operator $(-\Delta)^s$ is the fractional Laplacian which is the Fourier multiplier by $|\xi|^{2s}$. The fractional Schr\"odinger equation was discovered by N. Laskin \cite{Laskin} as a result of extending the Feynmann path integral, from the Brownian-like to L\'evy-like quantum mechanical paths. The fractional Schr\"odinger equation also appears in the study of  water waves equations (see e.g. Refs. \cite{IonescuPusateri, Nguyen}). The study of the nonlinear fractional Schr\"odinger equation has attracted a lot of interest in the last decade (see e.g. Refs. \cite{BoulengerHimmelsbachLenzmann, ChoHajaiejHwangOzawa, ChoHwangKwonLee, ChoHwangOzawa, Feng, FrankLenzmann, FrankLenzmannSilvestre, HongSire, IonescuPusateri, KleinSparberMarkowich, PengShi, Zhu} and references cited therein).\newline
\indent The equation $(\ref{focusing intercritical NLFS})$ enjoys the scaling invariance
\[
u_\lambda(t,x):= \lambda^{\frac{2s}{\alpha}} u(\lambda^{2s} t, \lambda x), \quad \lambda >0.
\]
A calculation shows
\[
\|u_\lambda(0)\|_{\dot{H}^\gamma} = \lambda^{\gamma +\frac{2s}{\alpha}-\frac{d}{2}}\|u_0\|_{\dot{H}^\gamma}.
\]
From this, we define the critical Sobolev exponent
\begin{align}
\sct:= \frac{d}{2}-\frac{2s}{\alpha}, \label{critical sobolev exponent}
\end{align}
as well as the critical Lebesgue exponent 
\begin{align}
\alct: = \frac{2d}{d-2\sct} = \frac{d\alpha}{2s}. \label{critical lebesgue exponent}
\end{align}
By definition, we have the Sobolev embedding $\dot{H}^{\sct} \hookrightarrow L^{\alct}$. The equation $(\ref{focusing intercritical NLFS})$ is called $L^2$-subcritical ($L^2$-critical or $L^2$-supercritical) if $\sct<0$ ($\sct=0$ or $\sct>0$) respectively. \newline
\indent The local well-posedness for $(\ref{focusing intercritical NLFS})$ in Sobolev spaces with non-radial initial data was studied in Ref. \cite{HongSire} (see also Ref. \cite{Dinh-fract}). In the non-radial setting, the unitary group $e^{-it(-\Delta)^s}$ enjoys Strichartz estimates (see Ref. \cite{ChoOzawaXia} or Ref. \cite{Dinh-fract}):
\begin{align*}
\|e^{-it(-\Delta)^s} \psi \|_{L^p(\R, L^q)} & \lesssim \||\nabla|^{\gamma_{p,q}} \psi\|_{L^2},
\end{align*}
where $(p,q)$ satisfies the Schr\"odinger admissible condition 
\[
p \in [2,\infty], \quad q \in [2, \infty), \quad (p,q,d) \ne (2,\infty, 2), \quad \frac{2}{p} + \frac{d}{q} \leq \frac{d}{2},
\]
and 
\[
\gamma_{p,q} = \frac{d}{2} - \frac{d}{q}-\frac{2s}{p}.
\]
It is easy to see that the condition $\frac{2}{p} + \frac{d}{q} \leq \frac{d}{2}$ implies $\gamma_{p,q}>0$ for all Schr\"odinger admissible pairs $(p,q)$ except $(p,q)=(\infty,2)$. This means that for non-radial data, Strichartz estimates for $e^{-it(-\Delta)^s}$ have a loss of derivatives except for $(p,q)=(\infty,2)$. This makes the study of local well-posedness in the non-radial case more difficult. The local theory for $(\ref{focusing intercritical NLFS})$ showed in Refs. \cite{HongSire, Dinh-fract} is much weaker than the one for classical nonlinear Schr\"odinger equation, i.e. $s=1$. In particular, in the $\dot{H}^s$-subcritical case (i.e. $\sct<s$) the equation $(\ref{focusing intercritical NLFS})$ is locally well-posed in $H^s$ only for dimensions $d=1,2,3$. The loss of derivatives in Strichartz estimates can be removed if one considers radial initial data. More precisely, we have for $d 
\geq 2$, $\frac{d}{2d-1}\leq s<1$ and $\psi$ radial,
\[
\|e^{-it(-\Delta)^s} \psi\|_{L^p(\R, L^q)} \lesssim \|\psi\|_{L^2},
\]
provided that $(p,q)$ satisfies the fractional admissible condition
\[
p \in [2,\infty], \quad q \in [2,\infty), \quad (p,q) \ne \left(2,\frac{4d-2}{2d-3}\right), \quad \frac{2s}{p} +\frac{d}{q} = \frac{d}{2}. 
\]
These Strichartz estimates with no loss of derivatives allow us to show a better local theory for $(\ref{focusing intercritical NLFS})$ with radial initial data. We refer the reader to Section $\ref{section preliminaries}$ for more details. \newline
\indent The existence of blowup solutions to $(\ref{focusing intercritical NLFS})$ was studied numerically in Ref. \cite{KleinSparberMarkowich}. Later, Boulenger-Himmelsbach-Lenzmann \cite{BoulengerHimmelsbachLenzmann} established blowup criteria for radial $H^s$ solutions to $(\ref{focusing intercritical NLFS})$. Note that in Ref. \cite{BoulengerHimmelsbachLenzmann}, they considered $H^{2s}$ solutions due to the lack of a full local theory at the time of consideration. Thanks to the local theory given in Section $\ref{section preliminaries}$, we can recover $H^s$ solutions by approximation arguments. More precisely, they proved the following: 
\begin{theorem}[Ref. \cite{BoulengerHimmelsbachLenzmann}] \label{theorem blowup criteria Hs} Let $d\geq 2$, $s \in (1/2,1)$ and $\alpha>0$. Let $u_0 \in H^s$ be radial and assume that the corresponding solution to $(\ref{focusing intercritical NLFS})$ exists on the maximal forward time interval $[0,T)$. 
	\begin{itemize}
		\item {\bf Mass-critical case}: If $\sce=0$ or $\alpha=\frac{4s}{d}$ and $E(u_0)<0$, then the solution $u$ either blows up in finite time, i.e. $T<+\infty$ or blows up infinite time, i.e. $T=+\infty$ and 
		\[
		\|u(t)\|_{\dot{H}^s} \geq c t^s, \quad \forall t\geq t_*,
		\]
		for some $C>0$ and $t_*>0$ depending only on $u_0, s$ and $d$. 
		\item {\bf Mass-supercritical and energy-subcritical case}: If $0<\sce<s$ or $\frac{4s}{d}<\alpha<\frac{4s}{d-2s}$ and $\alpha <4s$ and either $E(u_0)<0$, or if $E(u_0) \geq 0$, we assume that
		\[
		E^{\sce}(u_0) M^{s-\sce}(u_0) <E^{\sce}(Q) M^{s-\sce}(Q), \quad \|u_0\|^{\sce}_{\dot{H}^s} \|u_0\|^{s-\sce}_{L^2} > \|Q\|^{\sce}_{\dot{H}^s} \|Q\|^{s-\sce}_{L^2},
		\]
		where $Q$ is the unique (up to symmetries) positive radial solution to the elliptic equation
		\[
		(-\Delta)^s Q + Q - |Q|^\alpha Q=0,
		\]
		then the solution blows up in finite time, i.e. $T<+\infty$. 
		\item {\bf Energy-critical case}: If $\sce=s$ or $\alpha=\frac{4s}{d-2s}$ and $\alpha<4s$ and either $E(u_0)<0$, or if $E(u_0) \geq 0$, we assume that
		\[
		E(u_0)<E(W), \quad \|u_0\|_{\dot{H}^s} > \|W\|_{\dot{H}^s},
		\]
		where $W$ is the unique (up to symmetries) positive radial solution to the elliptic equation
		\[
		(-\Delta)^s W -|W|^{\frac{4s}{d-2s}} W=0,
		\]
		then the solution blows up in finite time, i.e. $T<+\infty$. 
	\end{itemize}
	Here $M(u)$ and $E(u)$ are the conserved mass and energy respectively. 
\end{theorem}
The blowup criteria of Boulenger-Himmelsbach-Lenzmann \cite{BoulengerHimmelsbachLenzmann} naturally lead to the study of dynamical properties such as blowup rate, concentration and limiting profile,.. of blowup solutions to $(\ref{focusing intercritical NLFS})$. \newline
\indent In the mass-critical case $\sct=0$ or $\alpha=\frac{4s}{d}$, the dynamics of blowup $H^s$ solutions was recently considered in Ref. \cite{Dinh-masscritical} (see also Ref. \cite{Feng}). The study of blowup $H^s$ solutions to the focusing mass-critical nonlinear fractional Schr\"odinger equation is connected to the notion of ground state which is the unique (up to symmetries) positive radial solution of the elliptic equation
\begin{align}
(-\Delta)^s Q + Q - |Q|^{\frac{4s}{d}}Q =0. \label{elliptic mass-critical}
\end{align}
Note that the existence and uniqueness (modulo symmetries) of ground state to $(\ref{elliptic mass-critical})$ were shown in Refs. \cite{FrankLenzmann, FrankLenzmannSilvestre}. Using the sharp Gagliardo-Nirenberg inequality
\[
\|f\|_{L^{\frac{4s}{d}+2}}^{\frac{4s}{d}+2} \leq C_{\text{GN}} \|f\|_{L^2}^{\frac{4s}{d}} \|f\|^2_{\dot{H}^s},
\]
with
\[
C_{\text{GN}} = \frac{2s+d}{d} \|Q\|^{-\frac{4s}{d}}_{L^2},
\]
the conservation of mass and energy show that if $u_0 \in H^s$ satisfies $\|u_0\|_{L^2} <\|Q\|_{L^2}$, then the corresponding solution exists globally in time. This suggests that $\|Q\|_{L^2}$ is the critical mass for formation of singularities. To study dynamical properties of blowup $H^s$ solutions to the mass-critical $(\ref{focusing intercritical NLFS})$, the author in Ref. \cite{Dinh-masscritical} proved a compactness lemma related to the equation by means of the profile decomposition for bounded sequences in $H^s$. 
\begin{proposition}[Compactness lemma \cite{Dinh-masscritical}] Let $d\geq 1$ and $0<s<1$. Let $(v_n)_{n\geq 1}$ be a bounded sequence in $H^s$ such that
	\[
	\limsup_{n\rightarrow \infty} \|v_n\|_{\dot{H}^s} \leq M, \quad \limsup_{n\rightarrow \infty} \|v_n\|_{L^{\frac{4s}{d}+2}} 
	\geq m.
	\]
	Then there exists a sequence $(x_n)_{n\geq 1}$ in $\R^d$ such that up to a subsequence,
	\[
	v_n(\cdot + x_n) \rightharpoonup V \text{ weakly in } H^s,
	\]
	for some $V \in H^s$ satisfying
	\[
	\|V\|_{L^2}^{\frac{4s}{d}} \geq \frac{d}{d+2s} \frac{m^{\frac{4s}{d}+2}}{M^2} \|Q\|^{\frac{4s}{d}}_{L^2}.
	\]
\end{proposition}
Thanks to this compactness lemma, the author in Ref. \cite{Dinh-masscritical} showed that the $L^2$-norm of blowup solutions must concentrate by an amount which is bounded from below by $\|Q\|_{L^2}$ at the blowup time. He also showed the limiting profile of blowup solutions with minimal mass $\|u_0\|_{L^2}=\|Q\|_{L^2}$, that is, up to symmetries of the equation, the ground state $Q$ is the profile for blowup solutions with minimal mass. \newline
\indent The main goal of this paper is to study dynamical properties of blowup solutions to $(\ref{focusing intercritical NLFS})$ in the mass-supercritical and energy-subcritical case with initial data in $\dot{H}^{\sct} \cap \dot{H}^s$. To this end, we first show the local well-posedness for $(\ref{focusing intercritical NLFS})$ with initial data in $\dot{H}^{\sct} \cap \dot{H}^s$. For data in $H^s$, the local well-posedness in non-radial and radial cases was showed in Refs. \cite{HongSire, Dinh-masscritical}. In the non-radial setting, the inhomogeneous Sobolev embedding $W^{s,q} \hookrightarrow L^r$ plays a crucial role (see e.g. Ref. \cite{HongSire}). Since we are considering data in $\dot{H}^{\sct} \cap \dot{H}^s$, the inhomogeneous Sobolev embedding does not help. We thus have to rely on Strichartz estimates without loss of derivatives and the homogeneous Sobolev embedding $\dot{W}^{s,q} \hookrightarrow L^r$. We hence restrict ourself to radially symmetric initial data, $d\geq 2$ and $\frac{d}{2d-1} \leq s<1$ for which Strichartz estimates without loss of derivatives are available. After the local theory is established, we show the existence of blowup $\dot{H}^{\sct} \cap \dot{H}^s$ solutions. The existence of blowup $H^s$ solutions for $(\ref{focusing intercritical NLFS})$ was shown in Ref. \cite{BoulengerHimmelsbachLenzmann} (see Theorem $\ref{theorem blowup criteria Hs}$). Note that the conservation of mass plays a crucial role in the argument of Ref. \cite{BoulengerHimmelsbachLenzmann}. In our consideration, the lack of mass conservation laws makes the problem more difficult. We are only able to show blowup criteria for negative energy intial data in $\dot{H}^{\sct} \cap \dot{H}^s$ with an additional assumption
\begin{align}
\sup_{t \in [0,T)} \|u(t)\|_{\dot{H}^{\sct}} <\infty, \label{bounded intro}
\end{align}
where $[0,T)$ is the maximal forward time of existence. In the mass-critical case $\sct=0$, this assumption holds trivially by the conservation of mass. We refer to Section $\ref{section preliminaries}$ for more details. To study blowup dynamics for data in $\dot{H}^{\sct} \cap \dot{H}^s$, we prove the profile decomposition for bounded sequences in $\dot{H}^{\sct} \cap \dot{H}^s$ which is proved by following the argument of Ref. \cite{HmidiKeraani} (see also Refs. \cite{Guoblowup, Dinh-blowupfourth}). This profile decomposition allows us to study the variational structure of the sharp constant to the Gagliardo-Nirenberg inequality
\begin{align}
\|f\|^{\alpha+2}_{L^{\alpha+2}} \leq A_{\text{GN}} \|f\|^{\alpha}_{\dot{H}^{\sct}} \|f\|^2_{\dot{H}^s}. \label{sharp gagliardo-nirenberg intro}
\end{align}
We will see in Proposition $\ref{prop variational structure sharp constants}$ that the sharp constant $A_{\text{GN}}$ is attained at a function $U \in \dot{H}^{\sct} \cap \dot{H}^s$ of the form
\[
U(x) = a Q(\lambda x + x_0),
\]
for some $a \in \C^*$, $\lambda>0$ and $x_0 \in \R^d$, where $Q$ is a solution to the elliptic equation
\begin{align*}
(-\Delta)^s Q +  (-\Delta)^{\sct} Q - |Q|^\alpha Q =0. 
\end{align*}
Moreover, 
\[
A_{\text{GN}} = \frac{\alpha+2}{2} \|Q\|^{-\alpha}_{\dot{H}^{\sct}}.
\]
The sharp Gagliardo-Nirenberg inequality $(\ref{sharp gagliardo-nirenberg intro})$ together with the conservation of energy yield the global existence for solutions satisfying
\[
\sup_{t\in [0,T)} \|u(t)\|_{\dot{H}^{\sct}} <\|Q\|_{\dot{H}^{\sct}}. 
\]
Another application of the profile decomposition is the compactness lemma, that is, for any bounded sequence $(v_n)_{n\geq 1}$ in $\dot{H}^{\sct} \cap \dot{H}^s$ satisfying
\[
\limsup_{n\rightarrow \infty} \|v_n\|_{\dot{H}^s} \leq M, \quad \limsup_{n\rightarrow \infty} \|v_n\|_{L^{\alpha+2}} \geq m,
\]
there exists a sequence $(x_n)_{n\geq 1}$ in $\R^d$ such that up to a subsequence,
\[
v_n(\cdot + x_n) \rightharpoonup V \text{ weakly in } \dot{H}^{\sct} \cap \dot{H}^s,
\]
for some $V \in \dot{H}^{\sct} \cap \dot{H}^s$ satisfying
\begin{align*}
\|V\|^\alpha_{\dot{H}^{\sct}} \geq \frac{2}{\alpha+2} \frac{ m^{\alpha+2}}{M^2} \|Q\|_{\dot{H}^{\sct}}^\alpha. 
\end{align*}
As a consequence, we show that the $\dot{H}^{\sct}$-norm of blowup solutions satisfying $(\ref{bounded intro})$ must concentrate by an amount which is bounded from below by $\|Q\|_{\dot{H}^{\sct}}$ at the blowup time (see Theorem $\ref{theorem concentration NLFS}$). We finally show in Theorem $\ref{theorem limiting profile critical norm}$ the limiting profile of blowup solutions with critical norm 
\begin{align}
\sup_{t\in [0,T)} \|u(t)\|_{\dot{H}^{\sct}} = \|Q\|_{\dot{H}^{\sct}}. \label{critical norm intro}
\end{align}
The paper is organized as follows. In Section $\ref{section preliminaries}$, we recall Strichartz estimates and show the local well-posednesss for data in $\dot{H}^{\sct} \cap \dot{H}^s$. We also prove blowup criteria for negative energy data in $\dot{H}^{\sct} \cap \dot{H}^s$ as well as the profile decomposition of bounded sequences in $\dot{H}^{\sct} \cap \dot{H}^s$. In Section $\ref{section variational analysis}$, we give some applications of the profile decomposition including the sharp Gagliardo-Nirenberg inequality $(\ref{sharp gagliardo-nirenberg intro})$ and the compactness lemma. In Section $\ref{section blowup concentration}$, we show the $\dot{H}^{\sct}$-concentration of blowup solutions. Finally, the limiting profile of blowup solutions with critical norm $(\ref{critical norm intro})$ will be given in Section $\ref{section limiting profile}$.
\section{Preliminaries} \label{section preliminaries}
\setcounter{equation}{0}
\subsection{Homogeneous Sobolev spaces}
We recall the definition of homogeneous Sobolev spaces needed in the sequel (see e.g. Refs. \cite{BerghLofstom}, \cite{GinibreVelo} or \cite{Triebel}). Denote $\mathcal{S}_0$ the subspace of the Schwartz space $\mathcal{S}$ consisting of functions $\phi$ satisfying $D^\beta \hat{\phi}(0) =0$ for all $\beta \in \N^d$, where $\hat{\cdot}$ is the Fourier transform on $\mathcal{S}$. Given $\gamma \in \R$ and $1 \leq q \leq \infty$, the generalized homogeneous Sobolev space $\dot{W}^{\gamma,q}$ is defined as a closure of $\mathcal{S}_0$ under the norm
\[
\|u\|_{\dot{W}^{\gamma,q}} := \||\nabla|^\gamma u\|_{L^q} <\infty.
\]
Under this setting, the spaces $\dot{W}^{\gamma,q}$ are Banach spaces. We shall use $\dot{H}^\gamma:= \dot{W}^{\gamma,2}$. Note that the spaces $\dot{H}^{\gamma_1}$ and $\dot{H}^{\gamma_2}$ cannot be compared for the inclusion. Nevertheless, for $\gamma_1 < \gamma < \gamma_2$, the space $\dot{H}^{\gamma}$ is an interpolation space between $\dot{H}^{\gamma_1}$ and $\dot{H}^{\gamma_2}$.
\subsection{Strichartz estimates}
We next recall Strichartz estimates for the fractional Schr\"odinger equation. To do so, we define for $I \subset \R$ and $p,q\in [1,\infty]$ the mixed norm 
\[
\|u\|_{L^p(I, L^q)} := \Big( \int_I \Big(\int_{\R^d} |u(t,x)|^q dx \Big)^{\frac{p}{q}} \Big)^{\frac{1}{p}},
\]
with a usual modification when either $p$ or $q$ are infinity. The unitary group $e^{-it(-\Delta)^s}$ enjoys several types of Strichartz estimates, for instance non-radial Strichartz estimates, radial Strichartz estimates and weighted Strichartz estimates (see e.g. Ref. \cite{ChoHajaiejHwangOzawa}). We only recall here two types: non-radial and radial Strichartz estimates. 
\begin{itemize}
	\item {\bf Non-radial Strichartz estimates (see e.g. Refs. \cite{ChoOzawaXia, Dinh-fract}):} for $d\geq 1$ and $s\in (0,1) \backslash \{1/2\}$, the following estimates hold:
	\begin{align*}
	\|e^{-it(-\Delta)^s} \psi\|_{L^p(\R, L^q)} & \lesssim \||\nabla|^{\gamma_{p,q}} \psi\|_{L^2}, \\
	\Big\| \int_0^t e^{-i(t-\tau)(-\Delta)^s} f(\tau) d\tau \Big\|_{L^p(\R, L^q)} &\lesssim \||\nabla|^{\gamma_{p,q}-\gamma_{a',b'}-2s} f\|_{L^{a'}(\R, L^{b'})},
	\end{align*}
	where $(p,q)$ and $(a,b)$ are Schr\"odinger admissible pairs, i.e.
	\[
	p \in [2,\infty], \quad q \in [2,\infty), \quad (p,q,d) \ne (2, \infty,2), \quad \frac{2}{p}+\frac{d}{q} \leq \frac{d}{2},
	\]
	and 
	\[
	\gamma_{p,q} = \frac{d}{2}-\frac{d}{q}-\frac{2s}{p},
	\]
	and similarly for $\gamma_{a',b'}$. As mentioned in the introduction, these Strichartz estimates have a loss of derivatives except for $(p,q)=(a,b)=(\infty,2)$. 
	\item {\bf Radial Strichartz estimates (see e.g. Refs. \cite{ChoLee}, \cite{GuoWang} or \cite{Ke}):} for $d\geq 2$ and $\frac{d}{2d-1} \leq s <1$, the following estimates hold:
	\begin{align}
	\|e^{-it(-\Delta)^s} \psi\|_{L^p(\R, L^q)} & \lesssim \| \psi\|_{L^2}, \label{radial strichartz homogeneous}\\
	\Big\| \int_0^t e^{-i(t-\tau)(-\Delta)^s} f(\tau) d\tau \Big\|_{L^p(\R, L^q)} &\lesssim \|f\|_{L^{a'}(\R, L^{b'})}, \label{radial strichartz inhomogeneous}
	\end{align}
	where $\psi$ and $f$ are radially symmetric and $(p,q), (a,b)$ sastisfy the fractional admissible condition:
	\begin{align}
	p\in [2,\infty], \quad q \in [2, \infty), \quad (p,q) \ne \left(2, \frac{4d-2}{2d-3} \right), \quad \frac{2s}{p}+\frac{d}{q} = \frac{d}{2}. \label{fractional admissible}
	\end{align}
\end{itemize}
\subsection{Local well-posedness} \label{subsection local well posedness}
In this subsection, we show the local well-posedness for $(\ref{focusing intercritical NLFS})$ with initial data in $\dot{H}^{\sct} \cap \dot{H}^s$. Before entering some details, let us recall the local well-posedness for $(\ref{focusing intercritical NLFS})$ with initial data in $H^s$. 
\begin{proposition} [Local well-posedness in $H^s$ \cite{Dinh-masscritical}] \label{prop local well-posedness Hs}
	Let
	\begin{align}
	\renewcommand{\arraystretch}{1.2}
	\left\{
	\begin{array}{llll}
	d=1, & \frac{1}{3}<s<\frac{1}{2}, & 0<\alpha<\frac{4s}{1-2s}, & u_0 \in H^s \text{ non-radial}, \\
	d=1, & \frac{1}{2}<s<1, & 0<\alpha<\infty, & u_0 \in H^s \text{ non-radial}, \\
	d=2, & \frac{1}{2}<s<1, & 0<\alpha<\frac{4s}{2-2s}, & u_0 \in H^s \text{ non-radial}, \\
	d=3, & \frac{3}{5} \leq s \leq \frac{3}{4}, & 0<\alpha<\frac{4s}{3-2s}, & u_0 \in H^s \text{ radial}, \\
	d=3, & \frac{3}{4}<s<1, & 0<\alpha<\frac{4s}{3-2s}, & u_0 \in H^s \text{ non-radial}, \\
	d\geq 4, & \frac{d}{2d-1} \leq s <1, & 0<\alpha<\frac{4s}{d-2s}, & u_0 \in H^s \text{ radial}.
	\end{array}
	\right.
	\label{condition LWP Hs}
	\end{align}
	Then the equation $(\ref{focusing intercritical NLFS})$ is locally well-posed in $H^s$. In addition, the maximal forward time of existence satisfies either $T=+\infty$ or $T<+\infty$ and $\lim_{t\uparrow T} \|u\|_{\dot{H}^s} =\infty$. Moreover, the solution enjoys the conservation of mass and energy, i.e. $M(u(t)) = M(u_0)$ and $E(u(t))=E(u_0)$ for all $t\in [0,T)$, where
	\begin{align*}
	M(u(t)) &= \int |u(t,x)|^2 dx, \\
	E(u(t)) &= \frac{1}{2} \int |(-\Delta)^{s/2} u(t,x)|^2 dx -\frac{1}{\alpha+2} \int |u(t,x)|^{\alpha+2}dx.
	\end{align*}
\end{proposition}
We now give the local well-posedness for $(\ref{focusing intercritical NLFS})$ with initial data in $\dot{H}^{\sct} \cap \dot{H}^s$.
\begin{proposition}[Local well-posedness in $\dot{H}^{\sct} \cap \dot{H}^s$] \label{prop LWP H dot s}
	Let $d\geq 2$, $\frac{d}{2d-1} \leq s <1$ and $\frac{4s}{d} \leq \alpha<\frac{4s}{d-2s}$. Let 
	\begin{align}
	p=\frac{4s(\alpha+2)}{\alpha(d-2s)}, \quad q=\frac{d(\alpha+2)}{d+\alpha s}. \label{choice of pq}
	\end{align}
	Then for any $u_0 \in \dot{H}^{\sce} \cap \dot{H}^s$ radial, there exist $T>0$ and a unique solution $u$ to $(\ref{focusing intercritical NLFS})$ satisfying
	\[
	u \in C([0,T), \dot{H}^{\sce} \cap \dot{H}^s) \cap L^p_{\emph{loc}}([0,T), \dot{W}^{\sce, q} \cap \dot{W}^{s,q}).
	\]
	The maximal forward time of existence satisfies either $T=+\infty$ or $T<+\infty$ and $\lim_{t\uparrow T} \|u(t)\|_{\dot{H}^{\sce}}+\|u(t)\|_{\dot{H}^s} =\infty$. Moreover, the solution enjoys the conservation of energy, i.e. $E(u(t)) = E(u_0)$ for all $t\in [0,T)$. 
\end{proposition}
\begin{remark} \label{rem LWP H dot s}
	When $\sct=0$ or $\alpha=\frac{4s}{d}$, Proposition $\ref{prop LWP H dot s}$ is a consequence of Proposition $\ref{prop local well-posedness Hs}$ since $\dot{H}^0 = L^2$ and $L^2 \cap \dot{H}^s= H^s$.
\end{remark}
\noindent \textit{Proof of Proposition $\ref{prop LWP H dot s}$.} It is easy to check that $(p,q)$ satisfies the fractional admissible condition $(\ref{fractional admissible})$. We next choose $(m,n)$ so that
\[
\frac{1}{p'}=\frac{1}{p}+\frac{\alpha}{m}, \quad \frac{1}{q'}=\frac{1}{q} + \frac{\alpha}{n}.
\]
We see that 
\[
\theta:=\frac{\alpha}{m}-\frac{\alpha}{p}=1-\frac{(d-2s)\alpha}{4s}>0, \quad q \leq n = \frac{dq}{d-sq}.
\]
The later fact ensures the Sobolev embedding $\dot{W}^{s,q} \hookrightarrow L^n$. Consider
\begin{align*}
X:= \left\{ u \in C(I, \dot{H}^{\sct} \cap \dot{H}^s) \cap L^p(I, \dot{W}^{\sct,q} \cap \dot{W}^{s,q}) \ : \ \right. & \|u\|_{L^\infty(I, \dot{H}^{\sct} \cap \dot{H}^s)}  \\
& \left. + \|u\|_{L^p(I, \dot{W}^{\sct,q} \cap \dot{W}^{s,q})} \leq M \right\},
\end{align*}
equipped with the distance
\[
d(u,v):= \|u-v\|_{L^\infty(I, L^2)} + \|u-v\|_{L^p(I, L^q)},
\]
where $I=[0, \zeta]$ and $M, \zeta>0$ to be determined later. Thanks to Duhamel's formula, it suffices to show that the functional
\[
\Phi(u)(t):= e^{-it(-\Delta)^s} u_0 + i \int_0^t e^{-i(t-\tau)(-\Delta)^s} |u(\tau)|^\alpha u(\tau) d\tau
\]
is a contraction on $(X,d)$. Thanks to Strichartz estimates $(\ref{radial strichartz homogeneous})$ and $(\ref{radial strichartz inhomogeneous})$, 
\begin{align*}
\|\Phi(u)\|_{L^\infty(I, \dot{H}^{\sct} \cap \dot{H}^s)} + \|\Phi(u)\|_{L^p(I, \dot{W}^{\sct,q} \cap \dot{W}^{s,q})} &\lesssim \|u_0\|_{\dot{H}^{\sct} \cap \dot{H}^s} + \||u|^\alpha u\|_{L^{p'}(I, \dot{W}^{\sct,q'} \cap \dot{W}^{s,q'})}, \\
\|\Phi(u)-\Phi(v)\|_{L^\infty(I, L^2)} + \|\Phi(u)-\Phi(v)\|_{L^p(I, L^q)} &\lesssim \||u|^\alpha u- |v|^\alpha v\|_{L^{p'}(I, L^{q'})}.
\end{align*}
By the fractional derivatives (see e.g. Proposition 3.1 of Ref. \cite{ChristWeinstein}) and the choice of $(m,n)$, the H\"older inequality implies
\begin{align*}
\||u|^\alpha u\|_{L^{p'}(I, \dot{W}^{\sct,q'} \cap \dot{W}^{s,q'})} &\lesssim \|u\|^\alpha_{L^m(I,L^n)} \|u\|_{L^p(I, \dot{W}^{\sct,q} \cap \dot{W}^{s,q})} \\
&\lesssim |I|^\theta \|u\|^\alpha_{L^p(I, L^n)} \|u\|_{L^p(I, \dot{W}^{\sct,q} \cap \dot{W}^{s,q})} \\
&\lesssim |I|^\theta \|u\|^\alpha_{L^p(I, \dot{W}^{s,q})} \|u\|_{L^p(I, \dot{W}^{\sct,q} \cap \dot{W}^{s,q})}.
\end{align*}
Similarly,
\begin{align*}
\||u|^\alpha u - |v|^\alpha v\|_{L^{p'}(I, L^{q'})} &\lesssim \left( \|u\|^\alpha_{L^m(I, L^n)} + \|v\|^\alpha_{L^m(I, L^n)}\right) \|u-v\|_{L^p(I, L^q)} \\
&\lesssim |I|^\theta \left( \|u\|^\alpha_{L^p(I, \dot{W}^{s,q})} + \|v\|^\alpha_{L^p(I, \dot{W}^{s,q})} \right) \|u-v\|_{L^p(I, L^q)}.
\end{align*}
This shows that for all $u,v \in X$, there exists $C>0$ independent of $\zeta$ and $u_0 \in \dot{H}^{\sct} \cap \dot{H}^s$ such that
\begin{align}
\|\Phi(u)\|_{L^\infty(I, \dot{H}^{\sct} \cap \dot{H}^s)} + \|\Phi(u)\|_{L^p(I, \dot{W}^{\sct,q} \cap \dot{W}^{s,q})}  &\leq C\|u_0\|_{\dot{H}^{\sct} \cap \dot{H}^s} + C\zeta^\theta M^{\alpha+1},  \label{blowup rate proof}\\
d(\Phi(u), \Phi(v)) &\leq C\zeta^\theta M^\alpha d(u,v). \nonumber
\end{align}
If we set $M=2C \|u_0\|_{\dot{H}^{\sct} \cap \dot{H}^s}$ and choose $\zeta>0$ so that
\[
C \zeta^\theta M^\alpha \leq \frac{1}{2},
\]
then $\Phi$ is a strict contraction on $(X,d)$. This proves the existence of solution 
\[
u \in C(I, \dot{H}^{\sct} \cap \dot{H}^s) \cap L^p(I, \dot{W}^{\sct,q} \cap \dot{W}^{s,q}).
\]
Note that by radial Strichartz estimates, the solution belongs to $L^a(I, \dot{W}^{\sct,b} \cap \dot{W}^{s,b})$ for any fractional admissible pairs $(a,b)$. The blowup alternative is easy since the time of existence depends only on the $\dot{H}^{\sct} \cap \dot{H}^s$-norm of initial data. The conservation of energy follows from the standard approximation. The proof is complete.
\defendproof
\begin{corollary}[Blowup rate] \label{coro blowup rate}
	Let $d\geq 2$, $\frac{d}{2d-1} \leq s <1$, $\frac{4s}{d} \leq \alpha<\frac{4s}{d-2s}$ and $u_0 \in \dot{H}^{\sce} \cap \dot{H}^s$ be radial. Assume that the corresponding solution $u$ to $(\ref{focusing intercritical NLFS})$ given in Proposition $\ref{prop LWP H dot s}$ blows up at finite time $0<T<+\infty$. Then there exists $C>0$ such that
	\begin{align}
	\|u(t)\|_{\dot{H}^{\sce} \cap \dot{H}^s} > \frac{C}{(T-t)^{\frac{s-\sce}{2s}}}, \label{blowup rate}
	\end{align}
	for all $0<t<T$.
\end{corollary}
\begin{proof}
	Let $0<t<T$. If we consider $(\ref{focusing intercritical NLFS})$ with initial data $u(t)$, then it follows from $(\ref{blowup rate proof})$ and the fixed point argument that if for some $M>0$,
	\[
	C\|u(t)\|_{\dot{H}^{\sct} \cap \dot{H}^s} + C(\zeta-t)^\theta M^{\alpha+1} \leq M,
	\]
	then $\zeta<T$. Thus, 
	\[
	C\|u(t)\|_{\dot{H}^{\sct} \cap \dot{H}^s} + C(T-t)^\theta M^{\alpha+1} > M,
	\]
	for all $M>0$. Choosing $M= 2C\|u(t)\|_{\dot{H}^{\sct} \cap \dot{H}^s}$, we see that
	\[
	(T-t)^\theta \|u(t)\|^\alpha_{\dot{H}^{\sct} \cap \dot{H}^s} >C.
	\]
	This implies
	\[
	\|u(t)\|_{\dot{H}^{\sct} \cap \dot{H}^s}> \frac{C}{(T-t)^{\frac{\theta}{\alpha}}},
	\]
	which is exactly $(\ref{blowup rate})$ since $\frac{\theta}{\alpha}= \frac{4s-\alpha(d-2s)}{4\alpha s} = \frac{s-\sct}{2s}$. The proof is complete.
\end{proof}
\subsection{Blowup criteria}
In this subsection, we prove blowup criteria for $\dot{H}^{\sct} \cap \dot{H}^s$ solutions to the mass-supercritical and energy-subcritical $(\ref{focusing intercritical NLFS})$. For initial data in $H^s$, Boulenger-Himmelsbach-Lenzmann proved blowup criteria for the equation (see Theorem $\ref{theorem blowup criteria Hs}$ for more details). The main difficulty in our consideration is that the conservation of mass is no longer available. We overcome this difficulty by assuming that the solution satisfies the uniform bound $(\ref{bounded intro})$. More precisely, we have the following:
\begin{proposition}[Blowup criteria] \label{prop blowup criteria}
	Let $d\geq 2$, $\frac{d}{2d-1} \leq s <1$, $\frac{4s}{d}<\alpha<\frac{4s}{d-2s}$ and $\alpha<4s$. Let $u_0 \in \dot{H}^{\sce} \cap \dot{H}^s$ be radial satisfying $E(u_0)<0$. Assume that the corresponding solution to $(\ref{focusing intercritical NLFS})$ defined on a maximal forward time interval $[0,T)$ satisfies $(\ref{bounded intro})$. Then the solution $u$ blows up in finite time, i.e. $T<+\infty$. 
\end{proposition}
\begin{remark}\label{rem blowup criteria}
	The condition $\alpha<4s$ comes from the radial Sobolev embedding (a analogous condition appears in Ref. \cite{BoulengerHimmelsbachLenzmann} (see again Theorem $\ref{theorem blowup criteria Hs}$)).
\end{remark}
\noindent \textit{Proof of Proposition $\ref{prop blowup criteria}$.} Let $\chi: [0,\infty) \rightarrow [0,\infty)$ be a smooth function such that 
\[
\chi(r)= \left\{
\begin{array}{ll}
r^2 & \text{if } r\leq 1, \\
0 & \text{if } r\geq 2,
\end{array}
\right.
\quad \text{and} \quad \chi''(r) \leq 2 \text{ for } r\geq 0.
\]
For a given $R>0$, we define the radial function $\chi_R: \R^d \rightarrow \R$ by
\[
\varphi_R(x)= \varphi_R(r):= R^2 \chi(r/R), \quad |x|=r.
\]
It is easy to see that
\[
2-\varphi''_R(r) \geq 0, \quad 2-\frac{\varphi'_R(r)}{r}  \geq 0, \quad 2d-\Delta \varphi_R(x) \geq 0, \quad \forall r \geq 0, \forall x \in \R^d.
\]
Moreover, 
\[
\|\nabla^j \varphi_R\|_{L^\infty} \lesssim R^{2-j}, \quad j=0,\cdots, 4,
\]
and
\[
\text{supp}(\nabla^j \varphi_R) \subset \left\{
\begin{array}{cl}
\{|x| \leq 2R \} &\text{for } j=1,2, \\
\{R \leq |x| \leq 2R\} &\text{for } j=3,4.
\end{array}
\right.
\]
Now let $u \in \dot{H}^{\sct} \cap \dot{H}^s$ be a solution to $(\ref{focusing intercritical NLFS})$. We define the local virial action by
\[
M_{\varphi_R}(t) := 2 \int \nabla \varphi_R(x) \cdot \im{(\overline{u}(t,x)\nabla u(t,x))} dx.
\]
The virial action $M_{\varphi_R}(t)$ is well-defined. Indeed, we first learn from the H\"older inequality and the Sobolev embedding $\dot{H}^{\sct} \hookrightarrow L^{\alct}$ that
\begin{align}
\|u\|_{L^2(|x| \lesssim R)} \lesssim R^{\sct} \|u\|_{L^{\alct}(|x| \lesssim R)} \lesssim R^{\sct} \|u\|_{\dot{H}^{\sct}(|x| \lesssim R)}. \label{estimate L2 norm}
\end{align}
Using the fact $\text{supp}(\nabla \varphi_R) \subset \{|x|\lesssim R\}$, $(\ref{estimate L2 norm})$ and the estimate given in Lemma A.1 of Ref. \cite{BoulengerHimmelsbachLenzmann}, we have 
\begin{align}
|M_{\varphi_R}(t)| &\leq C(\chi,R) \left( \||\nabla|^{\frac{1}{2}} u(t)\|_{L^2(|x| \lesssim R)}^2 + \|u(t)\|_{L^2(|x|\lesssim R)} \||\nabla|^{\frac{1}{2}} u(t)\|_{L^2(|x| \lesssim R)} \right) \nonumber \\
&\leq C(\chi,R)  \left( \|u(t)\|^{2-\frac{1}{s}}_{L^2(|x| \lesssim R)} \|u(t)\|_{\dot{H}^s(|x| \lesssim R)}^{\frac{1}{s}} + \|u(t)\|^{2-\frac{1}{2s}}_{L^2(|x|\lesssim R)} \|u(t)\|^{\frac{1}{2s}}_{\dot{H}^s(|x| \lesssim R)} \right) \label{well defined local virial action} \\
& \leq C(\chi,R) \left( \|u(t)\|^{2-\frac{1}{s}}_{\dot{H}^{\sct}(|x| \lesssim R)} \|u(t)\|_{\dot{H}^s(|x| \lesssim R)}^{\frac{1}{s}} + \|u(t)\|^{2-\frac{1}{2s}}_{\dot{H}^{\sct}(|x|\lesssim R)} \|u(t)\|^{\frac{1}{2s}}_{\dot{H}^s(|x| \lesssim R)} \right). \nonumber
\end{align}
This shows that $M_{\varphi_R}(t)$ is well-defined for all $t\in [0,T)$. Note that in the case $\chi(r)=r^2$ or $\varphi_R(x)=|x|^2$, we have formally the virial identity (see Lemma 2.1 of Ref. \cite{BoulengerHimmelsbachLenzmann}):
\begin{align}
M'_{|x|^2}(t) = 8s \|u(t)\|^2_{\dot{H}^s} - \frac{4d\alpha}{\alpha+2} \|u(t)\|^{\alpha+2}_{L^{\alpha+2}} = 4d\alpha E(u(t)) - 2(d\alpha-4s) \|u(t)\|^2_{\dot{H}^s}. \label{virial identity}
\end{align}
We also have from Lemma 2.1 of \cite{BoulengerHimmelsbachLenzmann} that for any $t\in [0,T)$,
\begin{align*}
M'_{\varphi_R}(t) &= -\int_0^\infty m^s \int \Delta^2 \varphi_R |u_m(t)|^2 dx dm + 4\sum_{j,k=1}^d \int_0^\infty m^s \int \partial^2_{jk} \varphi_R \partial_j \overline{u}_m(t) \partial_k u_m(t) dx dm \\
&\mathrel{\phantom{= -\int_0^\infty m^s \int \Delta^2 \varphi_R |u_m(t)|^2 dx dm}}- \frac{2\alpha}{\alpha+2} \int \Delta \varphi_R |u(t)|^{\alpha+2} dx,
\end{align*}
where
\begin{align*}
u_m(t):= c_s \frac{1}{-\Delta+m} u(t) = c_s \mathcal{F}^{-1} \left(\frac{\hat{u}(t)}{|\xi|^2 +m}\right), \quad m>0,
\end{align*}
with
\[
c_s:= \sqrt{\frac{\sin \pi s}{\pi}}.
\]
Since $\varphi_R(x)=|x|^2$ for $|x| \leq R$, we use $(\ref{virial identity})$ to write
\begin{align*}
M'_{\varphi_R}(t)& =8s \|u(t)\|^2_{\dot{H}^s} -\frac{4d\alpha}{\alpha+2} \|u(t)\|^{\alpha+2}_{L^{\alpha+2}} - 8s\|u(t)\|^2_{\dot{H}^s(|x|>R)} + \frac{4d\alpha}{\alpha+2} \|u(t)\|^{\alpha+2}_{L^{\alpha+2}(|x|>R)} \\
&\mathrel{\phantom{=8s \|u(t)\|^2_{\dot{H}^s}}} -\int_0^\infty m^s \int_{|x|>R} \Delta^2 \varphi_R |u_m(t)|^2 dx dm  \\
&\mathrel{\phantom{=8s \|u(t)\|^2_{\dot{H}^s}}}+ 4\sum_{j,k=1}^\infty \int_0^\infty m^s \int_{|x|>R} \partial^2_{jk} \varphi_R \partial_j \overline{u}_m(t) \partial_k u_m(t) dx dm \\
&\mathrel{\phantom{=8s \|u(t)\|^2_{\dot{H}^s} -\frac{4d\alpha}{\alpha+2} \|u(t)\|^{\alpha+2}_{L^{\alpha+2}}}}- \frac{2\alpha}{\alpha+2} \int_{|x|>R} \Delta \varphi_R |u(t)|^{\alpha+2} dx\\
&= 4d\alpha E(u(t)) - 2(d\alpha-4s) \|u(t)\|^2_{\dot{H}^s} \\
&\mathrel{\phantom{=}} + 4\sum_{j,k=1}^\infty \int_0^\infty m^s \int_{|x|>R} \partial^2_{jk} \varphi_R \partial_j \overline{u}_m(t) \partial_k u_m(t) dx dm - 8s \|u(t)\|^2_{\dot{H}^s(|x|>R)} \\
&\mathrel{\phantom{=}} - \int_0^\infty m^s \int_{|x|>R} \Delta^2 \varphi_R |u_m(t)|^2 dx dm + \frac{2\alpha}{\alpha+2} \int_{|x|>R} (2d-\Delta \varphi_R) |u(t)|^{\alpha+2} dx.
\end{align*}
Using
\[
\partial^2_{jk}= \left(\delta_{jk} - \frac{x_j x_k}{r^2} \right) \frac{\partial_r}{r} + \frac{x_j x_k}{r^2}\partial^2_r,
\]
we write
\[
4\sum_{j,k=1}^\infty \int_0^\infty m^s \int_{|x|>R} \partial^2_{jk} \varphi_R \partial_j \overline{u}_m(t) \partial_k u_m(t) dx dm = 4\int_0^\infty m^s \int_{|x|>R} \varphi''_R |\nabla u_m(t)|^2 dx dm.
\]
Note that (see (2.12) in Ref. \cite{BoulengerHimmelsbachLenzmann})
\[
\int_0^\infty m^s \int |\nabla f_m|^2 dx dm = \int \left(\frac{\sin \pi s}{\pi} \int_0^\infty \frac{m^s}{(|\xi|^2 +m)^2} dm \right) |\xi|^2 |\hat{f}(\xi)|^2 d\xi = s \|f\|^2_{\dot{H}^s}.
\]
We thus get
\begin{align*}
4\sum_{j,k=1}^\infty \int_0^\infty & m^s \int_{|x|>R} \partial^2_{jk} \varphi_R \partial_j \overline{u}_m(t) \partial_k u_m(t) dx dm \\
& = 8s\|u(t)\|^2_{\dot{H}^s(|x|>R)} -4 \int_0^\infty m^s \int_{|x|>R} (2-\varphi''_R) |\nabla u_m(t)|^2 dx dm \\
&\leq 8s \|u(t)\|^2_{\dot{H}^s(|x|>R)}.
\end{align*}
Thanks to Lemma A.2 of Ref. \cite{BoulengerHimmelsbachLenzmann}, the definition of $\varphi_R$ and the uniform bound $(\ref{bounded intro})$, we estimate
\begin{align*}
\Big| \int_0^\infty m^s \int_{|x|>R} \Delta^2 \varphi_R |u_m(t)|^2 dxdm \Big| &\lesssim \|\Delta^2 \varphi_R\|^s_{L^\infty} \|\Delta \varphi_R\|^{1-s}_{L^\infty} \|u\|^2_{L^2(|x|\lesssim R)} \\
& \lesssim R^{-2s} R^{2\sct} \|u(t)\|^2_{\dot{H}^{\sct}(|x|\lesssim R)} \lesssim R^{-2(s-\sct)}. 
\end{align*}
We thus obtain
\begin{align*}
M'_{\varphi_R}(t) &\leq 4d\alpha E(u(t)) - 2(d\alpha-4s)\|u(t)\|^2_{\dot{H}^s} + C R^{-2(s-\sct)} \\
&\mathrel{\phantom{\leq 4d\alpha E(u(t))}}+ \frac{2\alpha}{\alpha+2} \int_{|x|>R} (2d-\Delta \varphi_R) |u(t)|^{\alpha+2} dx.
\end{align*}
Since $\|2d-\Delta\varphi_R\|_{L^\infty} \lesssim 1$, it remains to bound $\|u(t)\|^{\alpha+2}_{L^{\alpha+2}(|x|>R)}$. To do this, we make use of the argument of Ref. \cite{MerleRaphael} (see also Ref. \cite{Dinh-blowupfourth}). Consider for $A>0$ the annulus $\mathcal{C} = \{A<|x| \leq 2A\}$, we claim that for any $\eps>0$,
\begin{align}
\|u(t)\|^{\alpha+2}_{L^{\alpha+2}(|x|>R)} \leq \eps \|u(t)\|_{\dot{H}^s}^2 + C(\eps) A^{-2(s-\sct)}. \label{estimate annulus}
\end{align}
To show $(\ref{estimate annulus})$, we recall the radial Sobolev embedding (see e.g. Ref. \cite{ChoOzawa}): 
\[
\sup_{x \ne 0} |x|^{\frac{d}{2}-\beta} |f(x)| \leq C(d, \beta) \|f\|_{\dot{H}^\beta},
\]
for all radial functions $f \in \dot{H}^\beta(\R^d)$ with $\frac{1}{2} <\beta<\frac{d}{2}$. Thanks to radial Sobolev embedding and $(\ref{estimate L2 norm})$, we have
\begin{align}
\|u(t)\|^{\alpha+2}_{L^{\alpha+2}(\mathcal{C})} &\lesssim \left( \sup_{\mathcal{C}} |u(t,x)|\right)^\alpha \|u(t)\|^2_{L^2(\mathcal{C})} & \nonumber \\
& \lesssim A^{-\left(\frac{d}{2}-\beta \right)\alpha} \|u(t)\|^\alpha_{\dot{H}^\beta(\mathcal{C})} \|u(t)\|^2_{L^2(\mathcal{C})}  &\left(\frac{1}{2}<\beta <\frac{d}{2} \right) \nonumber \\
&\lesssim A^{-\left(\frac{d}{2}-\beta \right)\alpha} \left( \|u(t)\|^{\frac{\beta}{s}}_{\dot{H}^s(\mathcal{C})} \|u(t)\|^{1-\frac{\beta}{s}}_{L^2(\mathcal{C})}\right)^\alpha \|u(t)\|^2_{L^2(\mathcal{C})} & \left(\frac{1}{2}<\beta<s<\frac{d}{2}\right) \nonumber \\
& \lesssim A^{-\left(\frac{d}{2}-\beta \right)\alpha} \|u(t)\|^{\frac{\alpha\beta}{s}}_{\dot{H}^s(\mathcal{C})} \|u(t)\|^{\left(1-\frac{\beta}{s}\right) \alpha +2}_{L^2(\mathcal{C})} & \nonumber \\
&\lesssim A^{-\vartheta} \|u(t)\|_{\dot{H}^s(\mathcal{C})}^{\frac{\alpha \beta}{s}}, & \label{estimate claim}
\end{align}
where
\[
\vartheta:= \left(\frac{d}{2}-\beta \right) \alpha - \left( \left(1-\frac{\beta}{s}\right) \alpha+2\right) \sct. 
\]
It is easy to check that 
\[
\vartheta= 2(s-\sct)\left( 1-\frac{\alpha\beta}{2s} \right).
\]
By our assumption $\alpha<4s$, we can choose $\frac{1}{2} <\beta<s$ so that $\vartheta>0$. We next apply the Young inequality to have for any $\eps>0$,
\[
A^{-\vartheta} \|u(t)\|^{\frac{\alpha \beta}{s}}_{\dot{H}^s(\mathcal{C})} \lesssim \eps \|u(t)\|^2_{\dot{H}^s(\mathcal{C})} + C(\eps) A^{-\frac{2s \vartheta}{2s-\alpha \beta}} = \eps \|u(t)\|^2_{\dot{H}^s(\mathcal{C})} + C(\eps) A^{-2(s-\sct)}.
\]
This combined with $(\ref{estimate claim})$ prove $(\ref{estimate annulus})$. We now write
\[
\int_{|x|>R} |u(t)|^{\alpha+2} dx = \sum_{j=0}^\infty \int_{2^j R<|x| \leq 2^{j+1}R} |u(t)|^{\alpha+2} dx,
\]
and apply $(\ref{estimate annulus})$ with $A=2^j R$ to get
\begin{align*}
\int_{|x|>R} |u(t)|^{\alpha+2} dx &\leq \eps \sum_{j=0}^\infty \|u(t)\|^2_{\dot{H}^s(2^j R < |x| \leq 2^{j+1}R)} + C(\eps) \sum_{j=0}^\infty (2^j R)^{-2(s-\sct)} \\
&\leq \eps \|u(t)\|^2_{\dot{H}^s(|x|>R)} + C(\eps) R^{-2(s-\sct)}.
\end{align*}
This shows that for any $\eps>0$,
\[
\|u(t)\|^{\alpha+2}_{L^{\alpha+2}(|x|>R)} \leq \eps \|u(t)\|^2_{\dot{H}^s(|x|>R)} + C(\eps) R^{-2(s-\sct)},
\]
and hence
\[
M'_{\varphi_R}(t) \leq 4d\alpha E(u(t)) - 2(d\alpha-4s) \|u(t)\|^2_{\dot{H}^s} + O\left(R^{-2(s-\sct)} + \eps \|u(t)\|^2_{\dot{H}^s} + C(\eps) R^{-2(s-\sct)} \right).
\]
By the conservation of energy with $E(u_0)<0$ and the fact $d\alpha>4s$, we take $\eps>0$ small enough and $R>0$ large enough to obtain
\begin{align}
M'_{\varphi_R}(t) \leq 2d\alpha E(u_0) -\delta \|u(t)\|^2_{\dot{H}^s}, \label{estimate local virial action}
\end{align}
where $\delta:= d\alpha-4s>0$. We now follow the argument of Ref. \cite{BoulengerHimmelsbachLenzmann}. Since $E(u_0)<0$, we learn from $(\ref{estimate local virial action})$ that $M'_{\varphi_R}(t) \leq -c$ for $c>0$. From this, we conclude that $M_{\varphi_R}(t)<0$ for all $t>t_1$ for some sufficiently large time $t_1 \gg 1$. Taking integration over $[t_1, t]$, we have
\begin{align}
M_{\varphi_R}(t) \leq -\delta \int_{t_1}^t \|u(\tau)\|^2_{\dot{H}^s} d\tau \leq 0, \quad \forall t \geq t_1. \label{local virial action bound}
\end{align}
We have from $(\ref{well defined local virial action})$ and the assumption $(\ref{bounded intro})$ that
\begin{align}
|M_{\varphi_R}(t)| \leq C(\chi,R) \left( \|u(t)\|^{\frac{1}{s}}_{\dot{H}^s} + \|u(t)\|^{\frac{1}{2s}}_{\dot{H}^s}\right). \label{bound local virial action}
\end{align}
We also have
\begin{align}
\|u(t)\|_{\dot{H}^s} \gtrsim 1, \quad \forall t\geq 0.  \label{lower bound}
\end{align}
Indeed, suppose it is not true. Then there exists a sequence $(t_n)_{n} \subset [0,+\infty)$ such that $\|u(t_n)\|_{\dot{H}^s} \rightarrow 0$ as $n \rightarrow \infty$. Thanks to the Gagliardo-Nirenberg inequality $(\ref{sharp gagliardo-nirenberg intro})$ and the assumption $(\ref{bounded intro})$, we see that $\|u(t_n)\|_{L^{\alpha+2}} \rightarrow 0$. We thus get $E(u(t_n)) \rightarrow 0$, which is a contradiction to $E(u(t)) = E(u_0)<0$. This shows $(\ref{lower bound})$. Combining $(\ref{bound local virial action})$ and $(\ref{lower bound})$, we obtain
\begin{align}
|M_{\varphi_R}(t)| \leq C(\chi,R) \|u(t)\|^{\frac{1}{s}}_{\dot{H}^s}. \label{bound local virial action 1}
\end{align}
Therefore, $(\ref{local virial action bound})$ and $(\ref{bound local virial action 1})$ yield
\[
M_{\varphi_R}(t) \leq C(\chi, R) \int_{t_1}^t |M_{\varphi_R}(\tau)|^{2s} d\tau, \quad \forall t\geq t_1. 
\]
By nonlinear integral inequality, we get
\[
M_{\varphi_R}(t) \lesssim C(\chi, R) |t-t_*|^{1-2s},
\]
for $s>1/2$ with some $t_*<+\infty$. Therefore, $M_{\varphi_R}(t) \rightarrow -\infty$ as $t\uparrow t_*$. Hence the solution cannot exist for all times $t\geq 0$. The proof is complete. 
\defendproof
\subsection{Profile decomposition}
In this subsection, we recall the profile decomposition for bounded sequences in $\dot{H}^{\sct} \cap \dot{H}^s$. 
\begin{theorem}[Profile decomposition] \label{theorem profile decomposition}
	Let $d\geq 1$, $0<s<1$ and $\frac{4s}{d}<\alpha<2^\star$, where
	\begin{align}
	2^\star:= \left\{
	\begin{array}{cl}
	\frac{4s}{d-2s} &\text{if } d>2s, \\
	\infty &\text{if } d \leq 2s.
	\end{array}
	\right. \label{define 2 star}
	\end{align}
	Let $(v_n)_{n\geq 1}$ be a bounded sequence in $\dot{H}^{\sce} \cap \dot{H}^s$. Then there exist a subsequence still denoted $(v_n)_{n\geq 1}$, a family $(x_n^j)_{j\geq 1}$ of sequences in $\R^d$ and a sequence  $(V^j)_{j\geq 1}$ of functions in $\dot{H}^{\sce} \cap \dot{H}^s$ such that
	\begin{itemize}
		\item for every $k\ne j$,
		\begin{align}
		|x_n^k - x_n^j| \rightarrow \infty, \quad \text{as } n \rightarrow \infty, \label{pairwise orthogonality}
		\end{align}
		\item for every $l\geq 1$ and every $x \in \R^d$,
		\[
		v_n(x) = \sum_{j=1}^l V^j(x-x_n^j) + v_n^l(x),
		\]
		with
		\begin{align}
		\limsup_{n\rightarrow \infty} \|v^l_n\|_{L^q} \rightarrow 0, \quad \text{as } l \rightarrow \infty, \label{profile error}
		\end{align}
		for every $q \in (\alce, 2+2^\star)$, where $\alce$ is given in $(\ref{critical lebesgue exponent})$. Moreover, 
		\begin{align}
		\|v_n\|^2_{\dot{H}^{\sce}} &= \sum_{j=1}^l \|V^j\|^2_{\dot{H}^{\sce}} + \|v^l_n\|^2_{\dot{H}^{\sce}} + o_n(1), \label{profile identity 1} \\
		\|v_n\|^2_{\dot{H}^s} &= \sum_{j=1}^l \|V^j\|^2_{\dot{H}^s} + \|v^l_n\|^2_{\dot{H}^s} + o_n(1), \label{profile identity 2}
		\end{align}
		as $n\rightarrow \infty$.
	\end{itemize}
\end{theorem}
\begin{remark} \label{rem profile decomposition}
	In the case $\sct=0$ or $\alpha=\frac{4s}{d}$, Theorem $\ref{theorem profile decomposition}$ is exactly Theorem 3.1 in Ref. \cite{Dinh-masscritical} due to the fact $\dot{H}^0=L^2$ and $L^2 \cap \dot{H}^s = H^s$.
\end{remark}
\noindent \textit{Proof of Theorem $\ref{theorem profile decomposition}$.} The proof is based on the argument of Ref. \cite{HmidiKeraani} (see also Refs. \cite{Guoblowup, Dinh-blowupfourth}). For reader's convenience, we give some details. Since $\dot{H}^{\sct} \cap \dot{H}^s$ is a Hilbert space, we denote $\Omega(v_n)$ the set of functions obtained as weak limits of sequences of the translated $v_n(\cdot + x_n)$ with $(x_n)_{n\geq 1}$ a sequence in $\R^d$. Set
\[
\eta(v_n):= \sup \{ \|v\|_{\dot{H}^{\sct}} + \|v\|_{\dot{H}^s} : v \in \Omega(v_n)\}.
\]
Clearly,
\[
\eta(v_n) \leq \limsup_{n\rightarrow \infty} \|v_n\|_{\dot{H}^{\sct}} + \|v_n\|_{\dot{H}^s}.
\]
We will show that there exist a sequence $(V^j)_{j\geq 1}$ of $\Omega(v_n)$ and a family $(x_n^j)_{j\geq 1}$ of sequences in $\R^d$ such that for every $k \ne j$,
\[
|x_n^k - x_n^j| \rightarrow \infty, 
\]
as $n \rightarrow \infty$ and up to a subsequence, we can write for every $l\geq 1$ and every $x \in \R^d$,
\[
v_n(x) = \sum_{j=1}^l V^j(x-x_n^j) + v^l_n(x), 
\]
with $\eta(v^l_n) \rightarrow 0$ as $l \rightarrow \infty$. Moreover, $(\ref{profile identity 1})$ and $(\ref{profile identity 2})$ hold as $n \rightarrow \infty$. \newline
\indent Indeed, if $\eta(v_n) =0$, then we take $V^j=0$ for all $j\geq 1$ and the proof is done. Otherwise we choose $V^1 \in \Omega(v_n)$ such that
\[
\|V^1\|_{\dot{H}^{\sct}} + \|V^1\|_{\dot{H}^s} \geq \frac{1}{2} \eta(v_n) >0. 
\]
By definition, there exists  a sequence $(x^1_n)_{n\geq 1}$ in $\R^d$ such that up to a subsequence,
\[
v_n(\cdot + x^1_n) \rightharpoonup V^1 \text{ weakly in } \dot{H}^{\sct} \cap \dot{H}^s.
\]
Set $v_n^1(x):= v_n(x) - V^1(x-x^1_n)$. It follows that $v^1_n(\cdot + x^1_n) \rightharpoonup 0$ weakly in $\dot{H}^{\sct} \cap \dot{H}^s$ and
thus
\begin{align*}
\|v_n\|^2_{\dot{H}^{\sct}} &= \|V^1\|^2_{\dot{H}^{\sct}} + \|v^1_n\|^2_{\dot{H}^{\sct}} + o_n(1),  \\
\|v_n\|^2_{\dot{H}^s} &= \|V^1\|^2_{\dot{H}^s} + \|v^1_n\|^2_{\dot{H}^s} + o_n(1),
\end{align*}
as $n \rightarrow \infty$. We next replace $(v_n)_{n\geq 1}$ by $(v^1_n)_{n\geq 1}$ and repeat the same argument. If $\eta(v^1_n) =0$, then we take $V^j=0$ for all $j \geq 2$ and the proof is done. Otherwise there exist $V^2 \in \Omega(v^1_n)$ and a sequence $(x^2_n)_{n\geq 1}$ in $\R^d$ such that
\[
\|V^2\|_{\dot{H}^{\sct}} + \|V^2\|_{\dot{H}^s} \geq \frac{1}{2} \eta(v^1_n)>0,
\]
and
\[
v^1_n(\cdot+x^2_n) \rightharpoonup V^2 \text{ weakly in } \dot{H}^{\sct} \cap \dot{H}^s.
\]
Set $v^2_n(x) := v^1_n(x) - V^2(x-x^2_n)$. It follows that
$v^2_n(\cdot +x^2_n) \rightharpoonup 0$ weakly in $\dot{H}^{\sct} \cap \dot{H}^s$ and 
\begin{align*}
\|v^1_n\|^2_{\dot{H}^{\sct}} & = \|V^2\|^2_{\dot{H}^{\sct}} + \|v^2_n\|^2_{\dot{H}^{\sct}} + o_n(1), \\
\|v^1_n\|^2_{\dot{H}^s} &= \|V^2\|^2_{\dot{H}^s} + \|v^2_n\|^2_{\dot{H}^s} + o_n(1),
\end{align*}
as $n \rightarrow \infty$. We now show that 
\[
|x^1_n - x^2_n| \rightarrow \infty, 
\]
as $n \rightarrow \infty$. Indeed, if it is not true, then up to a subsequence, $x^1_n - x^2_n \rightarrow x_0$ as $n \rightarrow \infty$ for some $x_0 \in \R^d$. Rewriting 
\[
v^1_n(x + x^2_n) = v^1_n(x +(x^2_n -x^1_n) + x^1_n),
\]
and using the fact $v^1_n (\cdot + x^1_n)$ converges weakly to $0$, we see that $V^2=0$. This implies that $\eta(v^1_n)=0$, which is a contradiction. An argument of iteration and orthogonal extraction allows us to construct the family $(x^j_n)_{j\geq 1}$ of sequences in $\R^d$ and the sequence $(V^j)_{j\geq 1}$ of functions in $\dot{H}^{\sct} \cap \dot{H}^s$ satisfying the claim above. Moreover, the convergence of the series $\sum_{j\geq 1}^\infty \|V^j\|^2_{\dot{H}^{\sct}} + \|V^j\|^2_{\dot{H}^s}$ implies that 
\[
\|V^j\|^2_{\dot{H}^{\sct}} + \|V^j\|^2_{\dot{H}^s} \rightarrow 0, \quad \text{as } j \rightarrow \infty.
\]
By construction,
\[
\eta(v^j_n) \leq 2 \left(\|V^{j+1}\|_{\dot{H}^{\sct}} + \|V^{j+1}\|_{\dot{H}^s}\right),
\]
which shows that $\eta(v^j_n) \rightarrow 0$ as $j \rightarrow \infty$. It remains to show $(\ref{profile error})$. To this end, we introduce for $R>1$ a function $\phi_R \in \mathcal{S}$ satisfying $\hat{\phi}_R: \R^d \rightarrow [0,1]$ and
\begin{align*}
\hat{\phi}_R(\xi) = \left\{
\begin{array}{clc}
1 &\text{if}& 1/R \leq |\xi| \leq R, \\
0 &\text{if}& |\xi| \leq 1/2R \vee |\xi|\geq 2R.
\end{array}
\right.
\end{align*}
We write
\[
v^l_n = \phi_R * v^l_n + (\delta - \phi_R) * v^l_n,
\]
where $\delta$ is the Dirac function and $*$ is the convolution operator. Let $q \in (\alct, 2+2^\star)$ be fixed. By Sobolev embedding and the Plancherel formula, 
\begin{align*}
\|(\delta -\phi_R) * v^l_n\|_{L^q} &\lesssim \|(\delta-\phi_R) * v^l_n\|_{\dot{H}^\beta} \lesssim \Big( \int |\xi|^{2\beta} |(1-\hat{\phi}_R(\xi)) \hat{v}^l_n(\xi)|^2 d\xi\Big)^{1/2} \\
&\lesssim \Big( \int_{|\xi|\leq 1/R} |\xi|^{2\beta} |\hat{v}^l_n(\xi)|^2 d\xi\Big)^{1/2} + \Big( \int_{|\xi|\geq R} |\xi|^{2\beta} |\hat{v}^l_n(\xi)|^2 d\xi\Big)^{1/2} \\
&\lesssim R^{\sct-\beta} \|v^l_n\|_{\dot{H}^{\sct}} + R^{\beta-s} \|v^l_n\|_{\dot{H}^s},
\end{align*}
where $\beta=\frac{d}{2}-\frac{d}{q} \in (\sct,s)$. Besides, the H\"older interpolation inequality yields
\begin{align*}
\|\phi_R * v^l_n\|_{L^q} &\lesssim \|\phi_R * v^l_n\|^{\frac{\alct}{q}}_{L^{\alct}} \|\phi_R * v^l_n\|^{1-\frac{\alct}{q}}_{L^\infty} \\
&\lesssim \|v^l_n\|^{\frac{\alct}{q}}_{\dot{H}^{\sct}} \|\phi_R * v^l_n\|^{1-\frac{\alct}{q}}_{L^\infty}.
\end{align*}
Observe that
\[
\limsup_{n\rightarrow \infty} \|\phi_R * v^l_n\|_{L^\infty} = \sup_{x_n} \limsup_{n\rightarrow \infty} |\phi_R * v^l_n(x_n)|.
\]
By the definition of $\Omega(v^l_n)$, we see that
\[
\limsup_{n\rightarrow \infty} \|\phi_R * v^l_n\|_{L^\infty} \leq \sup \Big\{ \Big| \int \phi_R(-x) v(x) dx\Big| : v \in \Omega(v^l_n)\Big\}.
\]
The Plancherel formula then implies
\begin{align*}
\Big|\int \phi_R(-x) v(x) dx \Big| &= \Big| \int \hat{\phi}_R(\xi) \hat{v}(\xi) d\xi\Big| \lesssim \|\|\xi|^{-\sct}\hat{\phi}_R\|_{L^2} \||\xi|^{\sct}\hat{v}\|_{L^2} \\
&\lesssim R^{\frac{d}{2}-\sct} \|\hat{\phi}_R\|_{\dot{H}^{-\sct}} \|v\|_{\dot{H}^{\sct}} \lesssim R^{\frac{2s}{\alpha}} \eta(v^l_n).
\end{align*}
Thus, for every $l\geq 1$,
\begin{align*}
\limsup_{n\rightarrow \infty} \|v^l_n\|_{L^q} &\lesssim \limsup_{n\rightarrow \infty} \|(\delta-\phi_R)* v^l_n\|_{L^q} + \limsup_{n\rightarrow \infty} \|\phi_R * v^l_n\|_{L^q} \\
&\lesssim R^{\sct-\beta} \|v^l_n\|_{\dot{H}^{\sct}} + R^{\beta-s} \|v^l_n\|_{\dot{H}^s} + \|v^l_n\|^{\frac{\alct}{q}}_{\dot{H}^{\sct}} \left[R^{\frac{2s}{\alpha}} \eta(v^l_n)\right]^{\left(1-\frac{\alct}{q}\right)}.
\end{align*}
Choosing $R= \left[\eta(v^l_n)^{-1}\right]^{\frac{\alpha}{2s}-\eps}$ for some $\eps>0$ small enough, we learn that 
\begin{align*}
\limsup_{n\rightarrow \infty} \|v^l_n\|_{L^q} &\lesssim \eta(v^l_n)^{(\beta-\sct)\left(\frac{\alpha}{2s}-\eps\right)} \|v^l_n\|_{\dot{H}^{\sct}} + \eta(v^l_n)^{(s-\beta)\left(\frac{\alpha}{2s}-\eps\right)} \|v^l_n\|_{\dot{H}^s} \\
&\mathrel{\phantom{\lesssim \eta(v^l_n)^{(\beta-\sct)\left(\frac{\alpha}{2s}-\eps\right)} \|v^l_n\|_{\dot{H}^{\sct}}}} +\eta(v^l_n)^{\eps \frac{2s}{\alpha} \left(1-\frac{\alct}{q}\right)} \| v^l_n\|_{\dot{H}^{\sct}}^{\frac{\alct}{q}}.
\end{align*}
Letting $l \rightarrow \infty$ and using the uniform boundedness of $(v^l_n)_{l\geq 1}$ in $\dot{H}^{\sct} \cap \dot{H}^s$ together with the fact that $\eta(v^l_n) \rightarrow 0$ as $l \rightarrow \infty$, we obtain 
\[
\limsup_{n \rightarrow \infty} \|v^l_n\|_{L^q} \rightarrow 0, \quad \text{as } l \rightarrow \infty.
\]
This completes the proof of Theorem $\ref{theorem profile decomposition}$.
\defendproof 
\section{Variational analysis} \label{section variational analysis}
\setcounter{equation}{0}
Let $d\geq 1$, $0<s<1$ and $\frac{4s}{d}<\alpha<2^\star$ where $2^\star$ is given in $(\ref{define 2 star})$. We consider the variational problems
\begin{align*}
A_{\text{GN}}&:=\max\{H(f): f \in \dot{H}^{\sct} \cap \dot{H}^s\}, & H(f)&:= \|f\|^{\alpha+2}_{L^{\alpha+2}} \div \left[ \|f\|^{\alpha}_{\dot{H}^{\sct}} \|f\|^2_{\dot{H}^s} \right], \\
B_{\text{GN}}&:= \max\{K(f): f \in L^{\alct}\cap \dot{H}^s \}, & K(f)&:= \|f\|^{\alpha+2}_{L^{\alpha+2}} \div \left[ \|f\|^{\alpha}_{L^{\alct}} \|f\|^2_{\dot{H}^s} \right].
\end{align*}
Here $A_{\text{GN}}$ and $B_{\text{GN}}$ are respectively sharp constants in the following Gagliardo-Nirenberg inequalities
\begin{align*}
\|f\|^{\alpha+2}_{L^{\alpha+2}} &\leq A_{\text{GN}} \|f\|^{\alpha}_{\dot{H}^{\sct}} \|f\|^2_{\dot{H}^s}, \\
\|f\|^{\alpha+2}_{L^{\alpha+2}} &\leq B_{\text{GN}} \|f\|^{\alpha}_{L^{\alct}} \|f\|^2_{\dot{H}^s}. 
\end{align*}
\begin{lemma} \label{lem maximizer NLFS}
	If $g$ and $h$ are maximizers of $H(f)$ and $K(f)$ respectively, then $g$ and $h$ satisfy
	\begin{align}
	A_{\emph{GN}} \|g\|^{\alpha}_{\dot{H}^{\sce}} (-\Delta)^s g + \frac{\alpha}{2} A_{\emph{GN}} \|g\|^{\alpha-2}_{\dot{H}^{\sct}}\|g\|^2_{\dot{H}^s} (-\Delta)^{\sce}g -\frac{\alpha+2}{2}  |g|^{\alpha} g&=0, \label{maximizer equation 1} \\
	B_{\emph{GN}} \|h\|^{\alpha}_{L^{\alce}} (-\Delta)^s h + \frac{\alpha}{2} B_{\emph{GN}} \|h\|^{\alpha-\alce}_{L^{\alce}} \|h\|^2_{\dot{H}^s}|h|^{\alce-2} h -\frac{\alpha+2}{2}  |h|^{\alpha} h&=0, \label{maximizer equation 2}
	\end{align}
	respectively.
\end{lemma}
\begin{proof}
	Since $g$ is a maximizer of $H$ in $\dot{H}^{\sct} \cap \dot{H}^s$, $g$ satisfies the Euler-Lagrange equation
	\[
	\frac{d}{d\eps}\Big|_{\vert \eps=0} H(g+\eps \phi) =0,
	\]
	for all $\phi \in \mathcal{S}_0$. A calculation shows
	\begin{align*}
	\left.\frac{d}{d\eps}\right|_{\eps=0} \|g+\eps \phi\|^{\alpha+2}_{L^{\alpha+2}} &= (\alpha+2)  \int \re{(|g|^\alpha g \overline{\phi})} dx,\\
	\left.\frac{d}{d\eps}\right|_{\eps=0} \|g+\eps \phi\|^\alpha_{\dot{H}^{\sct}} &= \alpha \|g\|^{\alpha-2}_{\dot{H}^{\sct}}\int \re{((-\Delta)^{\sct} g \overline{\phi})}dx, 
	\end{align*}
	and 
	\begin{align*}
	\left.\frac{d}{d\eps}\right|_{\eps=0} \|g+\eps \phi\|^2_{\dot{H}^s} =2 \int \re{((-\Delta)^s g \overline{\phi})} dx.
	\end{align*}
	We thus get
	\begin{align*}
	(\alpha+2) \|g\|^{\alpha}_{\dot{H}^{\sct}} \|g\|^2_{\dot{H}^s} |g|^\alpha g -  \alpha \|g\|^{\alpha+2}_{L^{\alpha+2}} \|g\|^{\alpha-2}_{\dot{H}^{\sct}} \|g\|^2_{\dot{H}^s} (-\Delta)^{\sct} g - 2 \|g\|^{\alpha+2}_{L^{\alpha+2}} \|g\|^\alpha_{\dot{H}^{\sct}} (-\Delta)^s g =0.
	\end{align*}
	Dividing by $2\|g\|^{\alpha}_{\dot{H}^{\sct}} \|g\|^2_{\dot{H}^s}$, we obtain $(\ref{maximizer equation 1})$. The proof of $(\ref{maximizer equation 2})$ is similar by using 
	\[
	\left.\frac{d}{d\eps}\right|_{\eps=0} \|h+\eps \phi\|^\alpha_{L^{\alct}} = \alpha \|h\|^{\alpha-\alct}_{L^{\alct}} \int \re{(|h|^{\alct-2} h \overline{\phi})}dx.
	\]
	The proof is complete.
\end{proof}
A first application of the profile decomposition given in Theorem $\ref{theorem profile decomposition}$ is the following variational structure of the sharp constants $A_{\text{GN}}$ and $B_{\text{GN}}$. 
\begin{proposition}[Variational structure of sharp constants] \label{prop variational structure sharp constants}
	Let $d\geq 1$, $0<s<1$ and $\frac{4s}{d}<\alpha<2^\star$.
	\begin{itemize}
		\item The sharp constant $A_{\emph{GN}}$ is attained at a function $U \in \dot{H}^{\sce} \cap \dot{H}^s$ of the form
		\[
		U(x) = a Q(\lambda x + x_0),
		\]
		for some $a \in \C^*, \lambda>0$ and $x_0 \in \R^d$, where $Q$ is a solution to the elliptic equation
		\begin{align}
		(-\Delta)^s Q +  (-\Delta)^{\sce} Q - |Q|^\alpha Q =0. \label{elliptic equation sobolev}
		\end{align}
		Moreover, 
		\[
		A_{\emph{GN}} = \frac{\alpha+2}{2} \|Q\|^{-\alpha}_{\dot{H}^{\sce}}.
		\]
		\item The sharp constant $B_{\emph{GN}}$ is attained at a function $V \in L^{\alce} \cap \dot{H}^s$ of the form 
		\[
		V(x) = b R(\mu x + y_0),
		\]
		for some $b \in \C^*, \mu>0$ and $y_0 \in \R^d$, where $R$ is a solution to the elliptic equation
		\begin{align}
		(-\Delta)^s R + |R|^{\alce-2} R - |R|^\alpha R =0. \label{elliptic equation lebesgue}
		\end{align}
		Moreover, 
		\[
		B_{\emph{GN}} = \frac{\alpha+2}{2} \|R\|^{-\alpha}_{L^{\alce}}.
		\]
	\end{itemize}
\end{proposition}
\begin{proof}
	We only prove Item 1, the proof for Item 2 is similar using the Sobolev embedding $\dot{H}^{\sct} \hookrightarrow L^{\alct}$. Observe that $H$ is invariant under the scaling
	\[
	f_{\mu, \lambda}(x) := \mu f(\lambda x), \quad \mu, \lambda>0.
	\]
	Indeed, a simple computation shows
	\[
	\|f_{\mu, \lambda} \|^{\alpha+2}_{L^{\alpha+2}} = \mu^{\alpha+2} \lambda^{-d} \|f\|^{\alpha+2}_{L^{\alpha+2}}, \quad \|f_{\mu, \lambda} \|_{\dot{H}^{\sct}}^\alpha 
	= \mu^\alpha \lambda^{-2s}\|f\|^\alpha_{\dot{H}^{\sct}}, \quad \|f_{\mu, \lambda}\|^2_{\dot{H}^s} = \mu^2 \lambda^{2s-d} \|f\|_{\dot{H}^s}^2.
	\]
	Thus, $H(f_{\mu, \lambda}) = H(f)$ for any $\mu, \lambda>0$. Moreover, if we set $g(x) = \mu f(\lambda x)$ with 
	\[
	\mu = \left(\frac{\|f\|^{\frac{d}{2}-s}_{\dot{H}^{\sct}}}{\|f\|_{\dot{H}^s}^{\frac{2s}{\alpha}}}\right)^{\frac{1}{s-\sct}}, \quad \lambda = \left(\frac{\|f\|_{\dot{H}^{\sct}}}{\|f\|_{\dot{H}^s}}\right)^{\frac{1}{s-\sct}},
	\]
	then $\|g\|_{\dot{H}^{\sct}} = \|g\|_{\dot{H}^s} = 1$ and $H(g) = H(f)$. Now let $(v_n)_{n\geq 1}$ be the maximizing sequence of $H$, i.e. $H(v_n) \rightarrow A_{\text{GN}}$ as $n\rightarrow \infty$. By scaling invariance, we may assume that $\|v_n\|_{\dot{H}^{\sct}} = \|v_n\|_{\dot{H}^s} =1$ and $H(v_n) = \|v_n\|^{\alpha+2}_{L^{\alpha+2}} \rightarrow A_{\text{GN}}$ as $n\rightarrow \infty$. It follows that $(v_n)_{n\geq 1}$ is bounded in $\dot{H}^{\sct} \cap \dot{H}^s$, and  the profile decomposition given in Theorem $\ref{theorem profile decomposition}$ shows that there exist a sequence $(V^j)_{j\geq 1}$ of $\dot{H}^{\sct} \cap \dot{H}^s$ functions and a family $(x^j_n)_{j\geq 1}$ of sequences in $\R^d$ such that up to a subsequence, 
	\[
	v_n(x) = \sum_{j=1}^l V^j(x-x^j_n) + v^l_n(x),
	\]
	and $(\ref{pairwise orthogonality}), (\ref{profile error})$, $(\ref{profile identity 1})$ and $(\ref{profile identity 2})$ hold. In particular, for any $l\geq 1$,
	\begin{align}
	\sum_{j=1}^l \|V^j\|^2_{\dot{H}^{\sct}} \leq 1, \quad \sum_{j=1}^l \|V^j\|^2_{\dot{H}^s} \leq 1, \label{bounded sequence variational structure}
	\end{align}
	and
	\[
	\limsup_{n\rightarrow \infty} \|v^l_n\|^{\alpha+2}_{L^{\alpha+2}} \rightarrow 0, \quad \text{as } l \rightarrow \infty. 
	\]
	Thus,
	\begin{align}
	A_{\text{GN}} &= \lim_{n\rightarrow \infty} \|v_n\|^{\alpha+2}_{L^{\alpha+2}} = \limsup_{n\rightarrow \infty} \Big\| \sum_{j=1}^l V^j(\cdot - x^j_n) + v^l_n \Big\|^{\alpha+2}_{L^{\alpha+2}} \nonumber \\
	&\leq \limsup_{n\rightarrow \infty} \Big( \Big\| \sum_{j=1}^l V^j(\cdot - x^j_n)  \Big\|_{L^{\alpha+2}} + \|v^l_n\|_{L^{\alpha+2}} \Big)^{\alpha+2} \nonumber \\
	&\leq \limsup_{n\rightarrow \infty} \Big\| \sum_{j=1}^\infty V^j(\cdot - x^j_n)  \Big\|_{L^{\alpha+2}}^{\alpha+2}. \label{infinite sum}
	\end{align}
	By the elementary inequality 
	\begin{align}
	\left| \Big| \sum_{j=1}^l a_j\Big|^{\alpha+2} - \sum_{j=1}^l |a_j|^{\alpha+2} \right| \leq C \sum_{j \ne k} |a_j| |a_k|^{\alpha+1}, \label{elementary inequality}
	\end{align}
	the pairwise orthogonality $(\ref{pairwise orthogonality})$ leads the mixed terms in the sum $(\ref{infinite sum})$ to vanish as $n\rightarrow \infty$. This shows that
	\[
	A_{\text{GN}} \leq \sum_{j=1}^\infty \|V^j\|^{\alpha+2}_{L^{\alpha+2}}.
	\]
	We also have from the definition of $A_{\text{GN}}$ that
	\[
	\frac{\|V^j\|_{L^{\alpha+2}}^{\alpha+2}}{A_{\text{GN}}} \leq \|V^j\|^{\alpha}_{\dot{H}^{\sct}} \|V^j\|^2_{\dot{H}^s},
	\]
	which implies
	\[
	1 \leq \frac{\sum_{j=1}^\infty \|V^j\|_{L^{\alpha+2}}^{\alpha+2}}{A_{\text{GN}}} \leq \sup_{j\geq 1} \|V^j\|^{\alpha}_{\dot{H}^{\sct}} \sum_{j=1}^\infty \|V^j\|^2_{\dot{H}^s}.
	\]
	Since $\sum_{j\geq 1} \|V^j\|^2_{\dot{H}^{\sct}}$ is convergent, there exists $j_0 \geq 1$ such that
	\[
	\|V^{j_0}\|_{\dot{H}^{\sct}} = \sup_{j\geq 1} \|V^j\|_{\dot{H}^{\sct}}. 
	\]
	By $(\ref{bounded sequence variational structure})$, we see that
	\[
	1 \leq \|V^{j_0}\|^{\alpha}_{\dot{H}^{\sct}} \sum_{j=1}^\infty \|V^j\|^2_{\dot{H}^s} \leq \|V^{j_0}\|_{\dot{H}^{\sct}}^{\alpha}. 
	\]
	It follows from $(\ref{bounded sequence variational structure})$ that $\|V^{j_0}\|_{\dot{H}^{\sct}}=1$ which shows that there is only one term $V^{j_0}$ is non-zero. Hence,
	\[
	\|V^{j_0}\|_{\dot{H}^{\sct}} = \|V^{j_0}\|_{\dot{H}^s} = 1, \quad \|V^{j_0}\|_{L^{\alpha+2}}^{\alpha+2} = A_{\text{GN}}. 
	\]
	It means that $V^{j_0}$ is the maximizer of $H$, and Lemma $\ref{lem maximizer NLFS}$ shows that
	\[
	A_{\text{GN}} (-\Delta)^s V^{j_0} + \frac{\alpha}{2} A_{\text{GN}}  (-\Delta)^{\sct} V^{j_0} -\frac{\alpha+2}{2}  |V^{j_0}|^{\alpha} V^{j_0}=0.
	\]
	Now if we set $V^{j_0}(x)= a Q(\lambda x+x_0)$ for some $a \in \C^*$, $\lambda>0$ and $x_0 \in \R^d$, then $Q$ solves $(\ref{elliptic equation sobolev})$ provided that
	\begin{align}
	|a| = \left(\frac{2 \lambda^{2s} A_{\text{GN}}}{\alpha+2}\right)^{\frac{1}{\alpha}}, \quad \lambda = \Big(\frac{\alpha}{2}\Big)^{\frac{1}{2(s-\sct)}}. \label{choice of a}
	\end{align}
	This shows the existence of solutions to $(\ref{elliptic equation sobolev})$. We next compute the sharp constant $A_{\text{GN}}$ in terms of $Q$. We have
	\[
	1=\|V^{j_0}\|^\alpha_{\dot{H}^{\sct}} = |a|^\alpha \lambda^{-2s} \|Q\|^\alpha_{\dot{H}^{\sct}} = \frac{2A_{\text{GN}}}{\alpha+2} \|Q\|^\alpha_{\dot{H}^{\sct}}.
	\]
	This implies $A_{\text{GN}} = \frac{\alpha+2}{2} \|Q\|^{-\alpha}_{\dot{H}^{\sct}}$. The proof is complete. 
\end{proof}
\begin{remark} \label{rem variational analysis}
	By $(\ref{choice of a})$ and the fact
	\begin{align*}
	1 = \|V^{j_0}\|^\alpha_{\dot{H}^{\sct}} &= |a|^\alpha \lambda^{-2s} \|Q\|^\alpha_{\dot{H}^{\sct}}, \\
	1 = \|V^{j_0}\|^2_{\dot{H}^s} &= |a|^2 \lambda^{2s-d} \|Q\|^2_{\dot{H}^s}, \\
	A_{\text{GN}} =\|V^{j_0}\|^{\alpha+2}_{L^{\alpha+2}} &= |a|^{\alpha+2} \lambda^{-d}\|Q\|^{\alpha+2}_{L^{\alpha+2}},
	\end{align*}
	we have the following Pohozaev identities
	\begin{align}
	\|Q\|^2_{\dot{H}^{\sct}} = \frac{\alpha}{2} \|Q\|^2_{\dot{H}^s} = \frac{\alpha}{\alpha+2} \|Q\|^{\alpha+2}_{L^{\alpha+2}}. \label{pohozaev identities}
	\end{align}
	The above identities can be showed by multiplying $(\ref{elliptic equation sobolev})$ with $\overline{Q}$ and $x \cdot  \nabla \overline{Q}$ and integrating over $\R^d$ and performing integration by parts. Indeed, multiplying $(\ref{elliptic equation sobolev})$ with $\overline{Q}$ and integrating by parts, we get
	\begin{align}
	\|Q\|^2_{\dot{H}^s}  + \|Q\|^2_{\dot{H}^{\sct}} - \|Q\|^{\alpha+2}_{L^{\alpha+2}} =0. \label{pohozaev equation 1}
	\end{align}
	Multiplying $(\ref{elliptic equation sobolev})$ with $x\cdot \nabla \overline{Q}$, integrating by parts and taking the real part, we have
	\begin{align}
	\Big(s-\frac{d}{2}\Big) \|Q\|^2_{\dot{H}^s} + \Big(\sct-\frac{d}{2}\Big) \|Q\|^2_{\dot{H}^{\sct}} + \frac{d}{\alpha+2} \|Q\|^{\alpha+2}_{L^{\alpha+2}} =0. \label{pohozaev equation 2}
	\end{align}
	From $(\ref{pohozaev equation 1})$ and $(\ref{pohozaev equation 2})$, we obtain $(\ref{pohozaev identities})$. Here we use the fact that for $\gamma \geq 0$,
	\begin{align*}
	\re{\int (-\Delta)^\gamma Q x \cdot \nabla \overline{Q} dx }  = \Big(\gamma-\frac{d}{2}\Big) \|Q\|^2_{\dot{H}^\gamma}. 
	\end{align*}
	The Pohozaev identities $(\ref{pohozaev identities})$ imply in particular that 
	\[
	H(Q)=\|Q\|^{\alpha+2}_{L^{\alpha+2}} \div \left[\|Q\|^\alpha_{\dot{H}^{\sct}} \|Q\|^2_{\dot{H}^s} \right] = \frac{\alpha+2}{2} \|Q\|^{-\alpha}_{\dot{H}^{\sct}} = A_{\text{GN}}, \quad E(Q)=0.
	\] 
	Similarly, we have
	\[
	\|R\|^2_{L^{\alct}} = \frac{\alpha}{2} \|R\|^2_{\dot{H}^s} = \frac{\alpha}{\alpha+2} \|R\|^{\alpha+2}_{L^{\alpha+2}}.
	\]
	In particular,
	\[
	K(R) = \|R\|^{\alpha+2}_{L^{\alpha+2}} \div \left[ \|R\|^\alpha_{L^{\alct}} \|R\|^2_{\dot{H}^s} \right] = \frac{\alpha+2}{2} \|R\|^{-\alpha}_{L^{\alct}} = B_{\text{GN}}, \quad E(R) = 0.
	\]
\end{remark}
\begin{definition}[Ground state] \label{definition ground state}
	\begin{itemize}
		\item We call \textbf{Sobolev ground states} the maximizers of $H$ which are solutions to $(\ref{elliptic equation sobolev})$. We denote the set of Sobolev ground states by $\Gc$. 
		\item We call \textbf{Lebesgue ground states} the maximizers of $K$ which are solutions to $(\ref{elliptic equation lebesgue})$. We denote the set of Lebesgue ground states by $\Hc$. 
	\end{itemize}
\end{definition}
Note that by Lemma $\ref{lem maximizer NLFS}$, if $g, h$ are  respectively Sobolev and Lebesgue ground states, then 
\[
A_{\text{GN}} = \frac{\alpha+2}{2} \|g\|^{-\alpha}_{\dot{H}^{\sct}}, \quad B_{\text{GN}} = \frac{\alpha+2}{2} \|h\|^{-\alpha}_{L^{\alct}}.
\]
This implies that Sobolev ground states have the same $\dot{H}^{\sct}$-norm, and all Lebesgue ground states have the same $L^{\alct}$-norm. Denote
\begin{align}
S_{\text{gs}}&:= \|g\|_{\dot{H}^{\sct}}, \quad \forall g \in \Gc, \label{critical sobolev norm} \\
L_{\text{gs}}&:= \|h\|_{L^{\alct}}, \quad \forall h \in \Hc. \label{critical lebesgue norm} 
\end{align}
In particular, we have the following sharp Gagliardo-Nirenberg inequalities
\begin{align}
\|f\|^{\alpha+2}_{L^{\alpha+2}} &\leq A_{\text{GN}} \|f\|^{\alpha}_{\dot{H}^{\sct}} \|f\|^2_{\dot{H}^s}, \label{sharp gagliardo-nirenberg inequality sobolev}\\
\|f\|^{\alpha+2}_{L^{\alpha+2}} &\leq B_{\text{GN}} \|f\|^{\alpha}_{L^{\alct}} \|f\|^2_{\dot{H}^s}, \label{sharp gagliardo-nirenberg inequality lebesgue}
\end{align}
with 
\[
A_{\text{GN}} = \frac{\alpha+2}{2} S_{\text{gs}}^{-\alpha}, \quad B_{\text{GN}} = \frac{\alpha+2}{2} L_{\text{gs}}^{-\alpha}.
\]
Another application of the profile decomposition given in Theorem $\ref{theorem profile decomposition}$ is the following compactness lemma.
\begin{theorem}[Compactness lemma] \label{theorem compactness lemma} Let $d\geq 1$, $0<s<1$ and $\frac{4s}{d}<\alpha<2^\star$. Let $(v_n)_{n\geq 1}$ be a bounded sequence in $\dot{H}^{\sce} \cap \dot{H}^s$ such that
	\[
	\limsup_{n\rightarrow \infty} \|v_n\|_{\dot{H}^s} \leq M, \quad \limsup_{n\rightarrow \infty} \|v_n\|_{L^{\alpha+2}} \geq m.
	\]
	\begin{itemize}
		\item Then there exists a sequence $(x_n)_{n\geq 1}$ in $\R^d$ such that up to a subsequence,
		\[
		v_n(\cdot + x_n) \rightharpoonup V \text{ weakly in } \dot{H}^{\sce} \cap \dot{H}^s,
		\]
		for some $V \in \dot{H}^{\sce} \cap \dot{H}^s$ satisfying
		\begin{align}
		\|V\|^\alpha_{\dot{H}^{\sce}} \geq \frac{2}{\alpha+2} \frac{ m^{\alpha+2}}{M^2} S_{\emph{gs}}^\alpha. \label{lower bound critical sobolev NLFS}
		\end{align}
		\item Then there exists a sequence $(y_n)_{n\geq 1}$ in $\R^d$ such that up to a subsequence,
		\[
		v_n(\cdot + y_n) \rightharpoonup W \text{ weakly in } L^{\alce} \cap \dot{H}^s,
		\]
		for some $W \in L^{\alce} \cap \dot{H}^s$ satisfying
		\begin{align}
		\|W\|^\alpha_{L^{\alce}} \geq \frac{2}{\alpha+2} \frac{ m^{\alpha+2}}{M^2} L_{\emph{gs}}^\alpha. \label{lower bound critical lebesgue NLFS}
		\end{align}
	\end{itemize}
\end{theorem}
\begin{remark} \label{rem compactness lemma}
	The lower bounds $(\ref{lower bound critical sobolev NLFS})$ and $(\ref{lower bound critical lebesgue NLFS})$ are optimal. In fact, if we take $v_n=Q \in \mathcal{G}$ in the first case and $v_n=R \in \mathcal{H}$ in the second case where $Q$ and $R$ are given in Proposition $\ref{prop variational structure sharp constants}$, then we get the equalities. 
\end{remark}
\noindent \textit{Proof of Theorem $\ref{theorem compactness lemma}$.} We only consider the first case, the second case is treated similarly using the Sobolev embedding $\dot{H}^{\sct} \hookrightarrow L^{\alct}$. By Theorem $\ref{theorem profile decomposition}$, there exist a sequence $(V^j)_{j\geq 1}$ of $\dot{H}^{\sct} \cap \dot{H}^s$ functions and a family $(x^j_n)_{j\geq 1}$ of sequences in $\R^d$ such that up to a subsequence, the sequence $(v_n)_{n\geq 1}$ can be written as
\[
v_n(x) = \sum_{j=1}^l V^j(x-x^j_n) + v^l_n(x),
\]
and $(\ref{profile error})$, $(\ref{profile identity 1})$ and $(\ref{profile identity 2})$ hold. This implies that
\begin{align}
m^{\alpha+2} &\leq \limsup_{n\rightarrow \infty} \|v_n\|_{L^{\alpha+2}}^{\alpha+2} = \limsup_{n\rightarrow \infty} \Big\| \sum_{j=1}^l V^j(\cdot -x^j_n) + v^l_n\Big\|^{\alpha+2}_{L^{\alpha+2}} \nonumber \\
&\leq \limsup_{n\rightarrow \infty} \Big( \Big\|\sum_{j=1}^l V^j(\cdot -x^j_n) \Big\|_{L^{\alpha+2}} + \|v^l_n\|_{L^{\alpha+2}}\Big)^{\alpha+2} \nonumber \\
&\leq \limsup_{n\rightarrow \infty} \Big\| \sum_{j=1}^\infty V^j(\cdot -x^j_n)\Big\|_{L^{\alpha+2}}^{\alpha+2}. \label{compactness lemma proof}
\end{align}
By the elementary inequality $(\ref{elementary inequality})$ and the pairwise orthogonality $(\ref{pairwise orthogonality})$, the mixed terms in the sum $(\ref{compactness lemma proof})$ vanish as $n\rightarrow \infty$. We thus get
\[
m^{\alpha+2} \leq \sum_{j=1}^\infty \|V^j\|_{L^{\alpha+2}}^{\alpha+2}.
\]
By the sharp Gagliardo-Nirenberg inequality $(\ref{sharp gagliardo-nirenberg inequality sobolev})$, we bound
\begin{align*}
\sum_{j=1}^\infty \|V^j\|^{\alpha+2}_{L^{\alpha+2}} \leq \frac{\alpha+2}{2} \frac{1}{S_{\text{gs}}^\alpha} \sup_{j\geq 1} \|V^j\|^{\alpha}_{\dot{H}^{\sct}} \sum_{j=1}^\infty \|V^j\|^2_{\dot{H}^s}. 
\end{align*}
By $(\ref{profile identity 2})$, we infer that
\[
\sum_{j=1}^\infty \|V^j\|^2_{\dot{H}^s}  \leq \limsup_{n\rightarrow \infty} \|v_n\|^2_{\dot{H}^s} \leq M^2.
\]
Therefore,
\[
\sup_{j\geq 1} \|V^j\|^{\alpha}_{\dot{H}^{\sct}} \geq \frac{2}{\alpha+2} \frac{m^{\alpha+2}}{M^2} S_{\text{gs}}^\alpha.
\]
Since the series $\sum_{j\geq 1} \|V^j\|^2_{\dot{H}^{\sct}}$ is convergent, the supremum above is attained. That is, there exists $j_0$ such that
\[
\|V^{j_0}\|^{\alpha}_{\dot{H}^{\sct}} \geq \frac{2}{\alpha+2} \frac{m^{\alpha+2}}{M^2} S_{\text{gs}}^\alpha.
\]
Rewriting
\[
v_n(x+ x^{j_0}_n) = V^{j_0} (x) + \sum_{1\leq j \leq l \atop j \ne j_0} V^j(x+ x_n^{j_0} - x^j_n) + \tilde{v}^l_n(x),
\]
with $\tilde{v}^l_n(x):= v^l_n(x+x^{j_0}_n)$, it follows from the pairwise orthogonality of the family $(x_n^j)_{j\geq 1}$ that
\[
V^j( \cdot +x^{j_0}_n -x^j_n) \rightharpoonup 0 \text{ weakly in } \dot{H}^{\sct} \cap \dot{H}^s,
\]
as $n \rightarrow \infty$ for every $j \ne j_0$. This shows that
\begin{align}
v_n(\cdot + x^{j_0}_n) \rightharpoonup V^{j_0} + \tilde{v}^l, \quad \text{as } n \rightarrow \infty, \label{compactness lemma proof 1}
\end{align}
where $\tilde{v}^l$ is the weak limit of $(\tilde{v}^l_n)_{n\geq 1}$. On the other hand, 
\[
\|\tilde{v}^l\|_{L^{\alpha+2}} \leq \limsup_{n\rightarrow \infty} \|\tilde{v}^l_n\|_{L^{\alpha+2}} = \limsup_{n\rightarrow \infty} \|v^l_n\|_{L^{\alpha+2}} \rightarrow 0, \quad \text{as } l \rightarrow \infty. 
\]
By the uniqueness of the weak limit $(\ref{compactness lemma proof 1})$, we get $\tilde{v}^l=0$ for every $l \geq j_0$. Therefore, we obtain 
\[
v_n(\cdot + x^{j_0}_n) \rightharpoonup V^{j_0}.
\]
The sequence $(x^{j_0}_n)_{n\geq 1}$ and the function $V^{j_0}$ now fulfill the conditions of Theorem $\ref{theorem compactness lemma}$. This ends the proof.
\defendproof  \newline
\indent We end this section by giving some applications of sharp Gagliardo-Nirenberg inequalities $(\ref{sharp gagliardo-nirenberg inequality sobolev})$ and $(\ref{sharp gagliardo-nirenberg inequality lebesgue})$.
\begin{proposition}[Global existence in $\dot{H}^{\sct} \cap \dot{H}^s$] \label{prop global existence 1}
	Let $d\geq 2$, $\frac{d}{2d-1} \leq s <1$ and $\frac{4s}{d}<\alpha<\frac{4s}{d-2s}$. Let $u_0 \in \dot{H}^{\sce} \cap \dot{H}^s$ be radial and the corresponding solution $u$ to $(\ref{focusing intercritical NLFS})$ defined on the maximal forward time interval $[0,T)$. Assume that
	\begin{align}
	\sup_{t\in [0,T)} \|u(t)\|_{\dot{H}^{\sce}} < S_{\emph{gs}}. \label{assumption global existence 1}
	\end{align}
	Then $T=+\infty$, i.e. the solution exists globally in time.
\end{proposition}
\begin{proof}
	Note that the assumption on $d, s$, $\alpha$ and $u_0$ comes from the local theory (see Section $\ref{section preliminaries}$). By the sharp Gagliardo-Nirenberg inequality $(\ref{sharp gagliardo-nirenberg inequality sobolev})$, we bound
	\begin{align*}
	E(u(t)) &=\frac{1}{2} \|u(t)\|^2_{\dot{H}^s} -\frac{1}{\alpha+2} \|u(t)\|^{\alpha+2}_{L^{\alpha+2}} \\
	&\geq \frac{1}{2} \left( 1- \Big( \frac{\|u(t)\|_{\dot{H}^{\sct}}}{S_{\text{gs}}}\Big)^\alpha\right) \|u(t)\|^2_{\dot{H}^s}.
	\end{align*}
	Thanks to the conservation of energy and the assumption $(\ref{assumption global existence 1})$, we obtain $\sup_{t\in [0,T)} \|u(t)\|_{\dot{H}^s} <\infty$. By the blowup alternative given in Proposition $\ref{prop LWP H dot s}$ and $(\ref{assumption global existence 1})$, the solution exists globally in time. The proof is complete.
\end{proof}
\begin{proposition} \label{prop global existence 2} 
	Let $d\geq 2$, $\frac{d}{2d-1} \leq s <1$ and $\frac{4s}{d}<\alpha<\frac{4s}{d-2s}$. Let $u_0 \in \dot{H}^{\sct} \cap \dot{H}^s$ be radial and the corresponding solution $u$ to $(\ref{focusing intercritical NLFS})$ defined on the maximal forward time interval $[0,T)$. Assume that
	\begin{align}
	S_{\emph{gs}} \leq \sup_{t\in [0,T)} \|u(t)\|_{\dot{H}^{\sct}} < \infty, \quad \sup_{t\in [0,T)} \|u(t)\|_{L^{\alce}} < L_{\emph{gs}}. \label{assumption global existence 2}
	\end{align}
	Then $T=+\infty$, i.e. the solution exists globally in time.
\end{proposition}
The proof is similar to the one of Proposition $\ref{prop global existence 1}$ by using the shap Gagliardo-Nirenberg inequality $(\ref{sharp gagliardo-nirenberg inequality lebesgue})$. 
\section{Blowup concentration} \label{section blowup concentration}
\setcounter{equation}{0}
\begin{theorem} [Blowup concentration] \label{theorem concentration NLFS}
	Let $d\geq 2$, $\frac{d}{2d-1} \leq s <1$ and $\frac{4s}{d}<\alpha<\frac{4s}{d-2s}$. Let $u_0 \in \dot{H}^{\sct} \cap \dot{H}^s$ be radial such that the corresponding solution $u$ to $(\ref{focusing intercritical NLFS})$ blows up at finite time $0<T<+\infty$. Assume that the solution satisfies $(\ref{bounded intro})$.
	Let $a(t)>0$ be such that 
	\begin{align}
	a(t) \|u(t)\|_{\dot{H}^s}^{\frac{1}{s-\sce}} \rightarrow \infty, \label{condition of a}
	\end{align}
	as $t\uparrow T$. Then there exist $x(t), y(t) \in \R^d$ such that
	\begin{align}
	\liminf_{t\uparrow T} \int_{|x-x(t)| \leq a(t)} |(-\Delta)^{\frac{\sce}{2}}u(t,x)|^2 dx \geq S_{\emph{gs}}^2, \label{sobolev concentration}
	\end{align}
	and
	\begin{align}
	\liminf_{t\uparrow T} \int_{|x-y(t)| \leq a(t)} |u(t,x)|^{\alce} dx \geq L_{\emph{gs}}^2. \label{lebesgue concentration}
	\end{align}
\end{theorem}
\begin{remark} \label{rem concentration NLFS}
	By the blowup rate given in Corollary $\ref{coro blowup rate}$ and the assumption $(\ref{bounded intro})$, we have 
	\[
	\|u(t)\|_{\dot{H}^s} > \frac{C}{(T-t)^{\frac{s-\sct}{2s}}},
	\]
	for $t\uparrow T$. Rewriting
	\begin{align*}
	\frac{1}{a(t) \|u(t)\|_{\dot{H}^s}^{\frac{1}{s-\sce}}} = \frac{\sqrt[2s]{T-t}}{a(t)} \frac{1}{\sqrt[2s]{T-t}\|u(t)\|_{\dot{H}^s}^{\frac{1}{s-\sce}}} &= \frac{\sqrt[2s]{T-t}}{a(t)} \left(\frac{1}{(T-t)^{\frac{s-\sct}{2s}} \|u(t)\|_{\dot{H}^s} } \right)^{\frac{1}{s-\sct}} \\
	&<C\frac{\sqrt[2s]{T-t}}{a(t)},
	\end{align*}
	we see that any function $a(t)>0$ satisfying $\frac{\sqrt[2s]{T-t}}{a(t)} \rightarrow 0$ as $t\uparrow T$ fulfills the conditions of Theorem $\ref{theorem concentration NLFS}$. 
\end{remark}
\noindent \textit{Proof of Theorem $\ref{theorem concentration NLFS}$.}
Let $(t_n)_{n\geq 1}$ be a sequence such that $t_n \uparrow T$ and $g\in \Gc$. Set
\[
\lambda_n := \left(\frac{\|g\|_{\dot{H}^s}}{\|u(t_n)\|_{\dot{H}^s}}\right)^{\frac{1}{s-\sct}}, \quad v_n(x):= \lambda_n^{\frac{2s}{\alpha}} u(t_n, \lambda_n x). 
\]
By the blowup alternative and the assumption $(\ref{bounded intro})$, we see that $\lambda_n \rightarrow 0$ as $n \rightarrow \infty$. Moreover, we have
\begin{align*}
\|v_n\|_{\dot{H}^{\sct}} = \|u(t_n)\|_{\dot{H}^{\sct}} <\infty,
\end{align*}
uniformly in $n$ and 
\[
\|v_n\|_{\dot{H}^s} = \lambda_n^{s-\sct}\|u(t_n)\|_{\dot{H}^s}= \|g\|_{\dot{H}^s},
\] 
and
\[
E(v_n) = \lambda_n^{2(s-\sct)} E(u(t_n)) = \lambda_n^{2(s-\sct)} E(u_0) \rightarrow 0, \quad \text{as } n \rightarrow \infty. 
\]
This implies in particular that
\begin{align*}
\|v_n\|^{\alpha+2}_{L^{\alpha+2}} \rightarrow \frac{\alpha+2}{2} \|g\|^2_{\dot{H}^s}, \quad \text{as } n \rightarrow \infty. 
\end{align*}
The sequence $(v_n)_{n\geq 1}$ satisfies the conditions of Theorem $\ref{theorem compactness lemma}$ with 
\[
m^{\alpha+2} = \frac{\alpha+2}{2} \|g\|^2_{\dot{H}^s}, \quad M^2 = \|g\|^2_{\dot{H}^s}. 
\]
Therefore, there exists a sequence $(x_n)_{n\geq 1}$ in $\R^d$ such that up to a subsequence,
\[
v_n(\cdot + x_n)  = \lambda_n^{\frac{2s}{\alpha}} u(t_n, \lambda_n \cdot + x_n) \rightharpoonup V \text{ weakly in } \dot{H}^{\sct} \cap \dot{H}^s,
\]
as $n \rightarrow \infty$ with $\|V\|_{\dot{H}^{\sct}} \geq  S_{\text{gs}}$. In particular, 
\[
(-\Delta)^{\frac{\sct}{2}} v(\cdot + x_n) = \lambda_n^{\frac{d}{2}} [(-\Delta)^{\frac{\sct}{2}} u](t_n, \lambda_n \cdot + x_n) \rightharpoonup (-\Delta)^{\frac{\sct}{2}} V \text{ weakly in } L^2. 
\]
This implies for every $R>0$,
\[
\liminf_{n\rightarrow \infty} \int_{|x|\leq R} \lambda_n^{d}| [(-\Delta)^{\frac{\sct}{2}} u](t_n, \lambda_n x + x_n)|^2 dx \geq \int_{|x|\leq R} |(-\Delta)^{\frac{\sct}{2}} V(x)|^2 dx,
\]
or 
\[
\liminf_{n\rightarrow \infty} \int_{|x-x_n|\leq R\lambda_n} | [(-\Delta)^{\frac{\sct}{2}} u](t_n, x)|^2 dx \geq \int_{|x|\leq R} |(-\Delta)^{\frac{\sct}{2}} V(x)|^2 dx.
\]
In view of the assumption $\frac{a(t_n)}{\lambda_n} \rightarrow \infty$ as $n\rightarrow \infty$, we get
\[
\liminf_{n\rightarrow \infty} \sup_{y\in \R^d} \int_{|x-y|\leq a(t_n)} |(-\Delta)^{\frac{\sct}{2}}u(t_n, x)|^2 dx \geq \int_{|x|\leq R} |(-\Delta)^{\frac{\sct}{2}}V(x)|^2 dx,
\]
for every $R>0$, which means that
\[
\liminf_{n\rightarrow \infty} \sup_{y \in \R^d} \int_{|x-y|\leq a(t_n)} |(-\Delta)^{\frac{\sct}{2}}u(t_n, x)|^2 dx \geq \int |(-\Delta)^{\frac{\sct}{2}}V(x)|^2 dx \geq  S_{\text{gs}}^2.
\]
Since the sequence $(t_n)_{n\geq 1}$ is arbitrary, we infer that
\[
\liminf_{t\uparrow T} \sup_{y\in \R^d} \int_{|x-y|\leq a(t)} |(-\Delta)^{\frac{\sct}{2}}u(t,x)|^2 dx \geq S_{\text{gs}}^2.
\]
But for every $t \in (0,T)$, the function $y\mapsto \int_{|x-y| \leq a(t)} |(-\Delta)^{\frac{\sct}{2}}u(t,x)|^2 dx$ is continuous and goes to zero at infinity. As a result, we get
\[
\sup_{y\in \R^d} \int_{|x-y|\leq a(t)} |(-\Delta)^{\frac{\sct}{2}}u(t,x)|^2 dx = \int_{|x-x(t)| \leq a(t)} |(-\Delta)^{\frac{\sct}{2}}u(t,x)|^2 dx,
\]
for some $x(t) \in \R^d$. This shows $(\ref{sobolev concentration})$. The proof for $(\ref{lebesgue concentration})$ is similar using Item 2 of Theorem $\ref{theorem compactness lemma}$. The proof is complete.
\defendproof
\section{Limiting profile with critical norms} \label{section limiting profile}
\setcounter{equation}{0}
Let us start with the following characterization of the ground state.
\begin{lemma} \label{lem characterization critical norm}
	Let $d\geq 1$, $0<s<1$ and $\frac{4s}{d}<\alpha<2^\star$. 
	\begin{itemize}
		\item If $u \in \dot{H}^{\sce} \cap \dot{H}^s$ is such that $\|u\|_{\dot{H}^{\sce}}=S_{\emph{gs}}$ and $E(u)=0$, then $u$ is of the form
		\[
		u(x) = e^{i\theta} \lambda^{\frac{2s}{\alpha}} g(\lambda x + x_0),
		\]
		for some $g\in \Gc$, $\theta \in \R, \lambda>0$ and $x_0 \in \R^d$. 
		\item If $u \in L^{\alce} \cap \dot{H}^s$ is such that $\|u\|_{L^{\alce}}=L_{\emph{gs}}$ and $E(u)=0$, then $u$ is of the form
		\[
		u(x) = e^{i\vartheta} \mu^{\frac{2s}{\alpha}} h(\mu x + y_0),
		\]
		for some $h\in \Hc$, $\vartheta \in \R, \mu>0$ and $y_0 \in \R^d$.
	\end{itemize}
\end{lemma}
\begin{proof}
	We only prove Item 1, Item 2 is treated similarly. Since $E(u)=0$, we have
	\[
	\|u\|^2_{\dot{H}^s} = \frac{2}{\alpha+2} \|u\|^{\alpha+2}_{L^{\alpha+2}}.
	\]
	Thus
	\[
	H(u) = \frac{\|u\|^{\alpha+2}_{L^{\alpha+2}}}{\|u\|^{\alpha}_{\dot{H}^{\sct}} \|u\|^2_{\dot{H}^s}} = \frac{\alpha+2}{2} \|u\|^{-\alpha}_{\dot{H}^{\sct}} = \frac{\alpha+2}{2} S_{\text{gs}}^{-\alpha} = A_{\text{GN}}.
	\]
	This shows that $u$ is the maximizer of $H$. Proposition $\ref{prop variational structure sharp constants}$ then implies that $u$ is of the form $u(x) = a g(\lambda x +x_0)$ for some $g\in \Gc$, $a \in \C^\star$, $\lambda>0$ and $x_0 \in \R^d$. Since $\|u\|_{\dot{H}^{\sct}} = S_{\text{gs}}=\|g\|_{\dot{H}^{\sct}}$, we have $|a|= \lambda^{\frac{2s}{\alpha}}$. The proof is complete.
\end{proof}
We are now able to show the limiting profile of blowup solutions with critical norms. 
\begin{theorem}[Limiting profile with critical norms] \label{theorem limiting profile critical norm}
	Let $d\geq 2$, $\frac{d}{2d-1} \leq s<1$ and $\frac{4s}{d} <\alpha<\frac{4s}{d-2s}$. Let $u_0 \in \dot{H}^{\sce} \cap \dot{H}^s$ be radial such that the corresponding solution $u$ to $(\ref{focusing intercritical NLFS})$ blows up at finite time $0<T<+\infty$. 
	\begin{itemize}
		\item Assume that 
		\begin{align}
		\sup_{t\in [0,T)} \|u(t)\|_{\dot{H}^{\sce}} = S_{\emph{gs}}. \label{assumption critical sobolev norm}
		\end{align}
		Then there exist $g \in \Gc$, $\theta(t)\in \R$, $\lambda(t)>0$ and $x(t) \in \R^d$ such that 
		\[
		e^{i\theta(t)}  \lambda^{\frac{2s}{\alpha}}(t) u(t, \lambda(t) \cdot + x(t)) \rightarrow g \text{ strongly in } \dot{H}^{\sce} \cap \dot{H}^s \text{ as } t \uparrow T.
		\]
		\item Assume that 
		\begin{align}
		\sup_{t\in [0,T)} \|u(t)\|_{\dot{H}^{\sce}} < \infty, \quad \sup_{t\in [0,T)} \|u(t)\|_{L^{\alce}} = L_{\emph{gs}}. \label{assumption critical lebesgue norm}
		\end{align}
		Then there exist $h\in \Hc$, $\vartheta(t)\in \R$, $\mu(t)>0$ and $y(t) \in \R^d$ such that 
		\[
		e^{i\vartheta(t)}  \mu^{\frac{2s}{\alpha}}(t) u(t, \mu(t) \cdot + y(t)) \rightarrow h \text{ strongly in } L^{\alce} \cap \dot{H}^s \text{ as } t \uparrow T.
		\]
	\end{itemize}
\end{theorem}
\begin{proof}
	We only prove the first item, the second one is treated similarly. We will show that for any $(t_n)_{n\geq 1}$ satisfying $t_n \uparrow T$, there exist a subsequence still denoted by $(t_n)_{n\geq 1}$, $g\in \Gc$, sequences of $\theta_n \in \R, \lambda_n>0$ and $x_n \in \R^d$ such that
	\begin{align}
	e^{i\theta_n} \lambda^{\frac{2s}{\alpha}}_n u(t_n, \lambda_n \cdot + x_n) \rightarrow g \text{ strongly in } \dot{H}^{\sct} \cap \dot{H}^s \text{ as } n \rightarrow \infty. \label{limiting profile critical norm proof}
	\end{align}
	Let $(t_n)_{n\geq 1}$ be a sequence such that $t_n \uparrow T$. Set
	\[
	\lambda_n := \left(\frac{\|Q\|_{\dot{H}^s}}{\|u(t_n)\|_{\dot{H}^s}}\right)^{\frac{1}{s-\sct}}, \quad v_n(x):= \lambda_n^{\frac{2s}{\alpha}} u(t_n, \lambda_n x),
	\]
	where $Q$ is as in Proposition $\ref{prop variational structure sharp constants}$. By the blowup alternative and $(\ref{assumption critical sobolev norm})$, we see that $\lambda_n \rightarrow 0$ as $n \rightarrow \infty$. Moreover, we have
	\begin{align}
	\|v_n\|_{\dot{H}^{\sct}} = \|u(t_n)\|_{\dot{H}^{\sct}} \leq S_{\text{gs}}=\|Q\|_{\dot{H}^{\sct}},  \label{property v_n}
	\end{align}
	and
	\begin{align}
	\|v_n\|_{\dot{H}^s} = \lambda_n^{s-\sct} \|u(t_n)\|_{\dot{H}^s} = \| Q\|_{\dot{H}^s}, \label{property v_n 1}
	\end{align}
	and
	\[
	E(v_n) = \lambda_n^{2(s-\sct)} E(u(t_n)) = \lambda_n^{2(s-\sct)} E(u_0) \rightarrow 0, \quad \text{as } n \rightarrow \infty. 
	\]
	This yields in particular that
	\begin{align}
	\|v_n\|^{\alpha+2}_{L^{\alpha+2}} \rightarrow \frac{\alpha+2}{2} \|Q\|^2_{\dot{H}^s}, \quad \text{as } n \rightarrow \infty. \label{convergence v_n}
	\end{align}
	The sequence $(v_n)_{n\geq 1}$ satisfies the conditions of Theorem $\ref{theorem compactness lemma}$ with 
	\[
	m^{\alpha+2} = \frac{\alpha+2}{2} \|Q\|^2_{\dot{H}^s}, \quad M^2 = \|Q\|^2_{\dot{H}^s}. 
	\]
	Therefore, there exists a sequence $(x_n)_{n\geq 1}$ in $\R^d$ such that up to a subsequence,
	\[
	v_n(\cdot + x_n)  = \lambda_n^{\frac{2s}{\alpha}} u(t_n, \lambda_n \cdot + x_n) \rightharpoonup V \text{ weakly in } \dot{H}^{\sct} \cap \dot{H}^s,
	\]
	as $n \rightarrow \infty$ with $\|V\|_{\dot{H}^{\sct}} \geq S_{\text{gs}}$. 
	Since $v_n(\cdot +x_n) \rightharpoonup V$ weakly in $\dot{H}^{\sct}\cap \dot{H}^s$ as $n\rightarrow \infty$, the semi-continuity of weak convergence and $(\ref{property v_n})$ imply
	\[
	\|V\|_{\dot{H}^{\sct}} \leq \liminf_{n\rightarrow \infty} \|v_n\|_{\dot{H}^{\sct}} \leq S_{\text{gs}}. 
	\]
	This together with the fact $\|V\|_{\dot{H}^{\sct}} \geq S_{\text{gs}}$ show that
	\begin{align}
	\|V\|_{\dot{H}^{\sct}} = S_{\text{gs}} = \lim_{n\rightarrow \infty} \|v_n\|_{\dot{H}^{\sct}}. \label{H dot gamma norm v_n}
	\end{align}
	Therefore,
	\[
	v_n(\cdot +x_n) \rightarrow V \text{ strongly in } \dot{H}^{\sct} \text{ as } n\rightarrow \infty. 
	\]
	On the other hand, the Gagliardo-Nirenberg inequality $(\ref{sharp gagliardo-nirenberg inequality sobolev})$ shows that $v_n(\cdot +x_n) \rightarrow V$ strongly in $L^{\alpha+2}$ as $n \rightarrow \infty$. Indeed, by $(\ref{property v_n 1})$,
	\begin{align*}
	\|v_n(\cdot +x_n) -V\|^{\alpha+2}_{L^{\alpha+2}} &\lesssim \|v_n(\cdot + x_n) - V\|^{\alpha}_{\dot{H}^{\sct}} \|v_n(\cdot + x_n) - V\|^2_{\dot{H}^s} \\
	&\lesssim (\|Q\|_{\dot{H}^s} + \|V\|_{\dot{H}^s})^2 \|v_n(\cdot +x_n) -V\|^{\alpha}_{\dot{H}^{\sct}} \rightarrow 0,
	\end{align*}
	as $n\rightarrow \infty$. Moreover, using $(\ref{convergence v_n})$ and $(\ref{H dot gamma norm v_n})$, the sharp Gagliardo-Nirenberg inequality $(\ref{sharp gagliardo-nirenberg inequality sobolev})$ yields
	\begin{align*}
	\|Q\|^2_{\dot{H}^s} = \frac{2}{\alpha+2} \lim_{n\rightarrow \infty} \|v_n\|^{\alpha+2}_{L^{\alpha+2}} = \frac{2}{\alpha+2} \|V\|^{\alpha+2}_{L^{\alpha+2}} \leq \Big(\frac{\|V\|_{\dot{H}^{\sct}}}{S_{\text{gs}}} \Big)^{\alpha} \|V\|^2_{\dot{H}^s} = \|V\|^2_{\dot{H}^s},
	\end{align*}
	or $\|Q\|_{\dot{H}^s} \leq \|V\|_{\dot{H}^s}$. By the semi-continuity of weak convergence and $(\ref{property v_n 1})$, 
	\[
	\|V\|_{\dot{H}^s} \leq \liminf_{n\rightarrow \infty} \|v_n\|_{\dot{H}^s} = \|Q\|_{\dot{H}^s}. 
	\]
	Therefore,
	\begin{align}
	\|V\|_{\dot{H}^s} = \|Q\|_{\dot{H}^s} = \lim_{n\rightarrow \infty} \|v_n\|_{\dot{H}^s}.\label{H dot 1 norm v_n}
	\end{align}
	Combining $(\ref{H dot gamma norm v_n}), (\ref{H dot 1 norm v_n})$ and using the fact $v_n(\cdot + x_n) \rightharpoonup V$ weakly in $\dot{H}^{\sct} \cap \dot{H}^s$, we conclude that 
	\[
	v_n(\cdot + x_n) \rightarrow V \text{ strongly in } \dot{H}^{\sct} \cap \dot{H}^s \text{ as } n\rightarrow \infty.
	\]
	In particular, we have $E(V) = \lim_{n\rightarrow \infty} E(v_n) =0$. This shows that there exists $V \in \dot{H}^{\sct} \cap \dot{H}^s$ such that
	\[
	\|V\|_{\dot{H}^{\sct}} = S_{\text{gs}}, \quad E(V) =0.
	\]
	By Lemma $\ref{lem characterization critical norm}$, there exists $g\in \Gc$ such that $V(x) = e^{i\theta} \lambda^{\frac{2s}{\alpha}} g(\lambda x +x_0)$ for some $\theta \in \R, \lambda>0$ and $x_0 \in \R^d$. Thus
	\[
	v_n(\cdot + x_n) = \lambda_n^{\frac{2s}{\alpha}} u(t_n, \lambda_n \cdot + x_n) \rightarrow V = e^{i\theta} \lambda^{\frac{2s}{\alpha}} g(\lambda \cdot +x_0) \text{ strongly in } \dot{H}^{\sct} \cap \dot{H}^s \text{ as } n \rightarrow \infty.
	\]
	Redefining variables as
	\[
	\tilde{\lambda}_n:= \lambda_n \lambda^{-1}, \quad \tilde{x}_n:= \lambda_n \lambda^{-1} x_0 +x_n,
	\]
	we get
	\[
	e^{-i\theta} \tilde{\lambda}^{\frac{2s}{\alpha}}_n u(t_n, \tilde{\lambda}_n \cdot + \tilde{x}_n) \rightarrow g \text{ strongly in } \dot{H}^{\sct} \cap \dot{H}^s \text{ as } n\rightarrow \infty.
	\]
	This proves $(\ref{limiting profile critical norm proof})$ and the proof is complete.
\end{proof}

\section*{Acknowledgments}
The author would like to express his deep thanks to his wife - Uyen Cong for her encouragement and support. He would like to thank his supervisor Prof. Jean-Marc Bouclet for the kind guidance and constant encouragement. He also would like to thank the reviewer for his/her helpful comments and suggestions.

\end{document}